%% file: homotopy.tex
\RequirePackage{fix-cm}
\documentclass[letterpaper,smallextended]{svjour3-reza}       
\smartqed  
\usepackage{enumerate}
\usepackage{framed}
\usepackage{amsfonts}
\usepackage{amssymb}
\usepackage{mathtools}
\usepackage{tikz}
\usetikzlibrary{%
	fadings,%
	shadings%
}
\usepackage{algorithm}
\usepackage{color,graphicx}
\usepackage{caption}
\usepackage{epstopdf}
\usepackage{subcaption}
\usepackage[margin=1.00in]{geometry}
\usepackage{multirow}
\usepackage{algpseudocode}
\providecommand{\normll}[1]{{\left\lVert#1\right\rVert}_{1,2}} 
\providecommand{\norml}[1]{{\left\lVert#1\right\rVert}_{1}}           
\providecommand{\normn}[1]{{\left\lVert#1\right\rVert}_{*}}          
\providecommand{\norm}[1]{{\left\lVert#1\right\rVert}}		  
\providecommand{\normop}[1]{{\left\vert\kern-0.25ex\left\vert\kern-0.25ex\left\vert #1 
    \right\vert\kern-0.25ex\right\vert\kern-0.25ex\right\vert}}		  
\providecommand{\normu}[1]{{\left\lVert#1\right\rVert}_2}		  
\providecommand{\dnorm}[1]{{\left\lVert#1\right\rVert}^*}	  
                                          
\providecommand{\inner}[2]{\langle {#1} , {#2} \rangle}		   
\providecommand{\innerB}[2]{\langle {#1} , {#2} \rangle} 	              
\providecommand{\card}{{\rm card}}      

\newcommand{\ptx}[2]{P_{{T}_{#1}}(#2)}
\newcommand{\ptpx}[2]{P_{{{T}_{#1}}^{\bot}}(#2)}
\newcommand\Algphase[1]{%
\vspace*{-.7\baselineskip}\Statex\hspace*{\dimexpr-\algorithmicindent-2pt\relax}\rule{\textwidth}{0.4pt}%
\Statex\hspace*{-\algorithmicindent}\text{#1}%
\vspace*{-.7\baselineskip}\Statex\hspace*{\dimexpr-\algorithmicindent-2pt\relax}\rule{\textwidth}{0.4pt}%
}
\DeclareMathOperator*{\argmin}{\arg\!\min}
\DeclareMathOperator*{\argmax}{\arg\!\max}
\DeclareMathOperator*{\trace}{\mathrm{trace}}

\DeclareMathOperator*{\vect}{vec}
\DeclareMathOperator*{\vspan}{span}
\newtheorem{thm}{Theorem}
\newtheorem{assumption}{Assumption}

\newtheorem{condition}{Condition}

\journalname{Comput Optim Appl}
\begin{document}

\title{Decomposable Norm Minimization with Proximal-Gradient Homotopy Algorithm\thanks{This material is based upon work supported by the National Science Foundation under Grant No. ECCS-0847077, and in part by the Office of Naval Research under Grant No. N00014-12-1-1002.}}


\author{Reza Eghbali         \and
        Maryam Fazel 
}


\institute{Maryam Fazel \at
              Department of Electrical Engineering, University of Washington, Seattle, WA 98195, USA\\
              \email{mfazel@uw.edu}           
           \and
              Reza Eghbali \at
              Department of Electrical Engineering, University of Washington, Seattle, WA 98195, USA\\
               \email{eghbali@uw.edu}    
}

\date{Received: 14 January 2015}

\maketitle

\begin{abstract}
\input{A1.tex}

\end{abstract}
\section{Introduction}
\label{intro}
\input{Intro.tex}
\input{Prelim.tex}
\input{DecomNorm.tex}

\input{Convergence.tex}

\input{Numerical.tex}

\input{appendix2.tex}


%
%

\begin{acknowledgements}
The authors are greatly indebted to Dr. Lin Xiao from Microsoft Research, Redmond, for his many valuable comments and suggestions. We thank Amin Jalali for his comments and helpful discussions.
\end{acknowledgements}

\bibliographystyle{spmpsci}      
\bibliography{myref}

\end{document}

%% file: A1.tex
We study the convergence rate of the proximal-gradient homotopy algorithm applied to norm-regularized linear least squares problems, for a general class of norms. The homotopy algorithm reduces the regularization parameter in a series of steps, and uses a proximal-gradient algorithm to solve the problem at each step. Proximal-gradient algorithm has a linear rate of convergence given that the objective function is strongly convex, and the gradient of the smooth component of the objective function is Lipschitz continuous. In many applications, the objective function in this type of problem is not strongly convex, especially when the problem is high-dimensional  and regularizers are chosen that induce sparsity or low-dimensionality.  
We show that if the linear sampling matrix satisfies certain assumptions and the regularizing norm is decomposable, proximal-gradient homotopy algorithm converges with a \emph{linear rate} even though the objective function is not strongly convex. Our result generalizes results on the linear convergence of homotopy algorithm for $l_1$-regularized least squares problems. Numerical experiments are presented that support the theoretical convergence rate analysis.


%% file: Intro.tex
In signal processing and statistical regression, problems arise in which the goal is to 
recover a structured model from a few, often noisy, linear measurements. Well studied examples 
include recovery of sparse vectors and low rank matrices. These problems can be 
formulated as non-convex optimization programs, which are computationally intractable in general.
One can relax these non-convex problems using appropriate convex penalty functions, for example $\ell_1$, $\ell_{1,2}$  and nuclear norms in sparse vector, group sparse and low rank matrix recovery problems. These relaxations perform very well in many practical applications. Following \cite{donoho2006compressed,candes2006near,candes2006stable}, there has been a flurry of publications that formalize the condition for recovery of sparse vectors, e.g., \cite{bunea2007sparsity,van2009conditions}, low rank matrices, e.g., \cite{recht2010guaranteed,candes2011tight,gross2011recovering} from linear measurements by solving the appropriate relaxed convex optimization problems. Alongside results for sparse vector and low 
rank matrix recovery several authors have proposed more general frameworks for 
structured model recovery problems with linear measurements \cite{candes2012simple,chandrasekaran2012convex,negahban2012unified}. In many problems of 
interest, to recover the model from linear noisy measurements, one can formulate the 
following optimization program:
\begin{align}\label{convex_program}
{\mbox{minimize}}&{  \quad \norm{x}}\\ \notag
\mbox{subject to}& \quad \normu{{A}x-b}^2 \leq \epsilon^2,
\end{align}
\noindent where $b \in \mathbb{R}^m$ is the measurements vector, $A \in \mathbb{R}^{m\times n}$ is 
the linear measurement matrix, $\epsilon^2$ is the noise energy and $\norm{\cdot}$ is a norm on $\mathbb{R}^n$ that promotes the desired structure in the solution. The regularized version of 
problem \eqref{convex_program} has the following form:
\begin{align}\label{regularized_program}
&{\operatorname{minimize}}{\quad \frac{1}{2} \normu{{A}x-b}^2
+\lambda \norm{x}},
\end{align}

\noindent where $\lambda>0$ is the regularization parameter.

There has been extensive work on algorithms for solving problem \eqref{convex_program} and \eqref{regularized_program} 
in special cases of $\ell_1$  and nuclear norms. First order methods have been the method of choice for large scale
problems, since each iteration is computationally cheap.  Of particular interest is the proximal-gradient method for minimization of composite functions, which are functions that can be written as sum of a differentiable convex function and a closed convex function. Proximal-gradient method can be utilized for solving the regularized problem \eqref{regularized_program}.



When the smooth component of the objective function has a Lipschitz continuous gradient, proximal-gradient algorithm has a convergence rate of $O(1/t)$, where $t$ is the iteration number. For the accelerated version of proximal-gradient algorithm, the convergence rate improves to $O(1/t^2)$. When the objective function is strongly convex as well, proximal-gradient has linear convergence, i.e. $O(\kappa^t)$ with $\kappa \in (0,1)$ \cite{nesterov2013gradient}. However, in instances of problem \eqref{regularized_program} that are of interest, the number of samples is less than the dimension of the space, hence the matrix $A$ has a non-zero null space which results in an objective function that is not strongly convex. Several algorithms that combine homotopy continuation over $\lambda$ with proximal-gradient steps have been proposed in the literature for problem \eqref{regularized_program} in the special cases of $\ell_1$ and nuclear norms \cite{hale2008fixed,wright2009sparse,wen2010fast,ma2011fixed,toh2010accelerated}. Xiao and Zhang \cite{xiao2013proximal} have studied an algorithm with homotopy with respect to $\lambda$ for solving $\ell_1 \text{ regularized least squares}$ problem. Formulating their algorithm based on Nesterov's proximal-gradient method, they have demonstrated that this algorithm has an overall linear rate 
of convergence whenever $A$ satisfies the restricted isometry 
property (RIP) and the final value of the regularizer parameter $\lambda$ is greater 
than a  problem-dependent
lower bound. 

\subsection{Our result}

We generalize the linear convergence rate analysis of the homotopy algorithm studied in \cite{xiao2013proximal} to problem \eqref{regularized_program} when the regularizing norm is decomposable, where decomposability is a condition introduced in \cite{candes2012simple}. In particular, $\ell_1$, $\ell_{1,2}$ and nuclear norms satisfy this condition. We derive properties for this class of norms that are used directly in the convergence analysis. These properties can independently be of interest. Among these properties is the sublinearity of the the function $K: \mathbb{R}^n \mapsto \{0,1,\ldots,n\}$, where $K$ is generalization of the notion of cardinality for decomposable norms. 


The linear convergence result holds under an assumption on the RIP constants of $A$, which in turn holds with high probability for several classes of random matrices when the number of measurements $m$ is large enough (orderwise the same as that required for recovery of the structured model). 


\subsection{Algorithms for structured model recovery}
There has been extensive work on algorithms for solving problems \eqref{convex_program} 
and \eqref{regularized_program} in the special cases of $\ell_1$ and nuclear norms. For a detailed review of first order methods we refer the reader to \cite{nesterov2013first} and references therein. In \cite{xiao2013proximal}, authors have reviewed sparse recovery and $\ell_1$ norm minimization algorithms that are related to the homotopy algorithm for $\ell_1$ norm. We discuss related algorithms mostly focusing on algorithms for other norms including nuclear norm here. 

Proximal-gradient method for $\ell_1$/nuclear norm minimization has a local linear convergence in a neighborhood of the optimal value \cite{hou2013linear,zhang2013linear,luo1992linear}. The proximal operator for nuclear norm is soft-thresholding operator on singular values. Several authors have proposed algorithms for low rank matrix recovery and matrix completion problem based on soft- or hard-thresholding operators; see, e.g.,  \cite{jain2010guaranteed,cai2010singular,mazumder2010spectral,ma2011fixed}. The singular value projection algorithm proposed by Jain et al. has a linear rate; however, to apply the hard-thresholding operator, one 
should know the rank of $x_0$. While the authors have introduced a heuristic for estimating the rank when it is not known a priori, their convergence results rely upon a known rank \cite{jain2010guaranteed}. SVP is the generalization of iterative hard thresholding algorithm (IHT) for sparse vector recovery. SVP and IHT belong to the family of greedy algorithms which do not solve a convex relaxation problem. Other greedy algorithms proposed for sparse recovery such as Compressive Sampling Matching Pursuit (CoSaMP) \cite{needell2009cosamp} and Fully Corrective Forward Greedy Selection (FCFGS) \cite{shalev2010trading} have also been generalized for recovery of general structured models including low-rank matrices and extended to more general loss functions \cite{nguyen2014linear,shalev2011large}.


For huge-scale problems with separable regularizing norm such as $\ell_1$ and $\ell_{1,2}$, coordinate descent methods can reduce the computational cost of each iteration significantly.  The convergence rate of randomized proximal coordinate descent method in expectation is orderwise the same as full proximal gradient descent; however, it can yield an improvement in terms of the dependence of convergence rate on $n$ \cite{nesterov2012efficiency,richtarik2014iteration,lu2013complexity}. To the best of our knowledge, linear convergence rate for any coordinate descent method applied to problem \eqref{convex_program} or \eqref{regularized_program} has not been shown in the literature.

Continuation over $\lambda$ for solving the regularized problem has been utilized in fixed point continuation algorithm (FPC) proposed by Ma et al. \cite{ma2011fixed} and accelerated proximal-gradient algorithm with line search (APGL) by Toh et al. \cite{toh2010accelerated}. FPC and APGL both solve a series of regularized problems where in each outer-iteration $\lambda$ is reduced by a factor less than one, the former uses soft-thresholding and the latter uses accelerated proximal-gradient for solving each regularized problem.

Agarwal et al. \cite{agarwal2011fast} have proposed algorithms for solving problems 
\eqref{convex_program} and \eqref{regularized_program} with an extra constraint in the 
form of $\norm{x} \leq \rho$. They have introduced the assumption of decomposability of the norm and given convergence analysis for norms that satisfy that assumption. They establish linear rate of convergence for their algorithms 
up to a neighborhood of the optimal solutions. However, their algorithm 
uses the bound $\rho$ which should be selected based on the norm of the 
true solution. In many problems this quantity is not known beforehand. Jin et al. \cite{jin2013new} have proposed an algorithm for $\ell_1 \text{ regularized least squares}$ that receives $\rho$ as a parameter and has linear rate of convergence. Their algorithm utilizes proximal gradient method but unlike homotopy algorithm reduces $\lambda$ at each step.

By using SDP formulation of nuclear norm, interior point methods can be utilized to solve problems \eqref{convex_program} and \eqref{regularized_program}. Interior point methods do not scale as well as first order methods for large scale problems (For example, for a general SDP solver when the dimension exceeds a few hundreds). However, Specialized SDP solvers for nuclear norm minimization can bring down the computational complexity of each iteration to $O(n^3)$ \cite{liu2009interior}.


%% file: Prelim.tex
\section{Preliminaries}
Let $A \in \mathbb{R}^{m \times n}$. We equip $\mathbb{R}^n$ by an inner product which is given by $\inner{x}{y} = x^T B y$ for some positive definite matrix $B$. We equip $\mathbb{R}^m$ with ordinary dot product $\inner{v}{u}= v^T u$. We denote the adjoint of $A$ with $A^* = B^{-1}A^T$. Note that for all $x \in \mathbb{R}^n$ and $u \in \mathbb{R}^m$
\begin{equation}
\inner{Ax}{u} = \inner{x}{A^* u}.
\end{equation}

We use $\normu{\cdot}$ to denote the norms induced by the inner product in  $\mathbb{R}^n$ and $\mathbb{R}^m$, that is:
\begin{align*}
\forall x\in \mathbb{R}^n:& \quad \normu{x} = \sqrt{x^T B x},\\
\forall v\in \mathbb{R}^m:& \quad \normu{v} = \sqrt{v^T v}.
\end{align*}


We use $\norm{\cdot}$ and $\dnorm{\cdot}$ to denote a regularizing norm and its dual on $\mathbb{R}^n$. The latter is defined as:
$$\dnorm{y} = \sup{\{\inner{y}{x}\, | \, \norm{x} \leq 1\}}.$$

Given a convex function $f: \mathbb{R}^n \mapsto \mathbb{R}$, $\partial f\left(x\right)$ denotes the set of subgradients of $f$ at $x$, i.e., the set of all $z\in \mathbb{R}^n$ such that
$$\forall y \in \mathbb{R}^n: \quad f(y) \geq f(x) + \inner{z}{y-x}.$$

When $f$ is differentiable, $\partial f\left(x\right) = \{\nabla f(x)\}$. Note that $\xi \in \partial \norm{x}$ if and only if
\begin{align}\label{norm-sub}
&\inner{\xi}{x} = \norm{x}, \\
&\dnorm{\xi} \leq 1.
\end{align}

We say $f$
is strongly convex with strong convexity parameter $\mu_f$ when $f\left(x\right)-\frac{\mu_f}{2} \normu{x}^2$ is convex. For a differentiable function this implies that for all $x,y \in \mathbb{R}^n$:
\begin{equation}\label{str1}
f\left(y\right) \geq f\left(x\right) + \langle \nabla f\left(x\right) , y-x \rangle + \frac{\mu_f}{2} \normu{x-y}^2.
\end{equation}

We call the gradient of a differentiable function Lipschitz continuous with Lipschitz constant $L_f$, when for all $x,y \in \mathbb{R}^n$:
\begin{equation}\label{lip2}
\normu{\nabla f\left(x\right)-\nabla f\left(y\right)} \leq L_f \normu{y-x}.
\end{equation}

For a convex function $f$, gradient Lipschitz continuity is equivalent to the following inequality [see \cite{nesterov2004introductory} Lemma 1.2.3. and Theorem 2.1.5]:
\begin{equation}\label{lip1}
f\left(y\right) \leq f\left(x\right) + \langle \nabla f\left(x\right) , y-x \rangle + \frac{L_f}{2} \normu{x-y}^2,
\end{equation}
\noindent for all $x,y \in \mathbb{R}^n$.

%% file: DecomNorm.tex
\section{Properties of the regularizing norm and $A$}\label{sec:properties}

In this section we introduce our assumptions on the regularizing norm $\norm{\cdot}$, and derive the properties of the norm based on these assumptions. 
The homotopy algorithm of \cite{xiao2013proximal} for the $\ell_1$-regularized problem is designed so that the  iterates maintain low cardinality throughout the algorithm, therefore one can use the restricted eigenvalue property of $A$, when $A$ acts on these iterates. Said another way, the squared loss term behaves like a strongly convex function over the algorithm iterates, which is why the algorithm can achieve a fast convergence rate. 
In the proof, \cite{xiao2013proximal} uses the the structure of the subdifferential of the $\ell_1$ norm,
\begin{align*}
\partial \norml{x} = \{{\rm sgn}(x) + v ~|~ v_i= 0 \text{  when } x_i \neq 0, \; \norm{v}_{\infty}\leq 1 \},
\end{align*}
as well as the following properties that hold for the cardinality function,
\begin{align*}
\norml{x}^2 &\leq \card(x) \normu{x}^2,\\
\card(x+y)  &\leq \card(x) + \card(y) \quad \text{(sublinearity).}  
\end{align*}

We first give our assumption on the structure of the subdifferential of a class norms (which inlcudes $\ell_1$ and nuclear norms but is much more general), and then derive the rest of the properties needed for generalizing the results of \cite{xiao2013proximal}. 

Before stating our assumptions, we add some more definitions to our tool box. Let $S^{n-1} = \{x \in \mathbb{R}^{n}\,|\,\normu{x}= 1\}$, and let $\mathcal{G}_{\norm{\cdot}}$ be the set of extreme points of the norm ball $\mathcal{B}_{\norm{\cdot}}:= \{x ~|~ \norm{x} \leq 1\}$. We impose two conditions on the regularizing norm.

\begin{condition}\label{condition-2} For any $x \in \mathcal{G}_{\norm{\cdot}}$, $\normu{x} = 1$, i.e., all the extreme points of the norm ball have unit $\normu{\cdot}$-norm.
\end{condition}

The second condition on the norm is the decomposability condition introduced in  \cite{candes2012simple}, which was inspired by the assumption introduced in \cite{negahban2012unified}. 

\begin{condition}[Decomposability]\label{condition-1}
For all $x \in \mathbb{R}^n$, there exists a subspace $T_x$ and a vector $e_x \in T_x$ such that
\begin{align}\label{decomposable}
\partial \norm{x} = \{e_x+v~|~ v \in {T_x^{\bot}}, \dnorm{v} \leq 1\}.
\end{align}
\end{condition}
Note that $x \in T_x$ for all $x \in \mathbb{R}^n$ because if $x \notin T_x$, then $x = y+z$ with $y \in T_x$ and $z \in T_x^{\bot}-\{0\}$. Let $z' = z / \dnorm{z}$. Since $e_x + z' \in \partial \norm{x}$, $\norm{x} = \inner{e_x + z'}{y + z} = \norm{x} + \normu{z}^2/\dnorm{z}$, which is a contradiction.

\vspace{0.1 in}
The decomposability condition has been used in both \cite{candes2012simple} and \cite{negahban2012unified} to give a simpler and unified proof for recovery of several structures such as sparse vectors and low-rank matrices. 

When attempting to extend this algorithm to general norms, several challenges arise. First, what is the appropriate generalization of cardinality for other structures and their corresponding norms? Essentially, we would need to count the number of nonzero coefficients in an appropriate representation and ensure there is a small number of nonzero coefficients in our iterates, to be able to apply a similar proof idea as in \cite{xiao2013proximal}.

The next theorem captures one of our main results for any decomposable norm. This theorem provides a new set of conditions that is based on the geometry of the norm ball, and we show are equivalent to decomposability on $\mathbb{R}^n$.  As a result, one can find a decomposition for any vector in $\mathbb{R}^n$ in terms of an orthogonal subset of $\mathcal{G}_{\norm{\cdot}}$.

\begin{thm}[Orthogonal representation]\label{OrRep}
Suppose $\mathcal{G}_{\norm{\cdot}} \subset S^{n-1}$, then $\norm{\cdot}$ is decomposable if and only if for any $x \in \mathbb{R}^n-\{0\}$ and $a_1 \in \argmax_{a \in \mathcal{G}_{\norm{\cdot}}}{\inner{a}{x}}$ there exist $a_2, \ldots, a_k \in  \mathcal{G}_{\norm{\cdot}}$ such that $\{a_1, a_2, \ldots, a_k\}$ is an orthogonal set that satisfies the following conditions:

\begin{itemize} \label{I}
\item[I] There exists $\{\gamma_i > 0 |\, i= 1,\ldots, k\}$ such that:
\begin{align}\notag
x = \sum_{i=1}^{k}{\gamma_i a_i}, \\ \label{norm-sum}
\norm{x} = \sum_{i=1}^{k}{\gamma_i}.
\end{align}
\label{II}
\item[II] For any set $\{\eta_i\,|\, |\eta_i| \leq 1,  i=1,\ldots, k\}$:
\begin{equation}\label{arbit-sum} 
\dnorm{\sum_{i=1}^{k}{\eta_i a_i}} \leq 1.
\end{equation}
\end{itemize}
Moreover, if $\{a_1, a_2, \ldots, a_k\} \subset \mathcal{G}_{\norm{\cdot}}$ satisfy I and II, then $e_x = \sum_{i=1}^{k} a_i$.
\end{thm}

The proof of Theorem \ref{OrRep} is presented in Appendix B. 

\vspace{0.1 in}
We will see in section \ref{sec::conv} that we need an orthogonal representation for all vectors to be able to bound the number of nonzero coefficients throughout the algorithm. First, we define a quantity $K(x)$ that bounds the ratio of the norm $\|\cdot\|$ to the Euclidean norm, and plays the same role in our analysis as cardinality played in \cite{xiao2013proximal}. Then we show that $K(x)$ is a sublinear function, that is, $K(x+y)\leq K(x)+K(y)$ for all $x,y$. This is a key property that is needed in the convergence analysis. Define $K: \mathbb{R}^n \mapsto \{0,1, 2, \ldots, n\}$

$$K\left(x\right) = \normu{e_x}^2.$$

Note that for every $x \in \mathbb{R}^n$,
\begin{align}\label{norm-ratio}
\norm{x}^2  = \inner{e_x}{x}^2 \leq  \normu{e_x}^2 \normu{x}^2= K(x) \normu{x}^2.
\end{align}

Here, the first equality follows from \eqref{norm-sub}, and the inequality follows from the Cauchy-Schwarz inequality.  
 In the analysis of homotopy algorithm we utilize \eqref{norm-ratio} alongside the structure of the subgradient given by \eqref{decomposable}.
$\ell_1$, $\ell_{1,2}$, and nuclear norms are three important examples that satisfy conditions \ref{condition-2} and \ref{condition-1}. Here we briefly discuss each one of these norms. 
\begin{itemize}

 \item{\bf Nuclear norm} on $\mathbb{R}^{d_1\times d_2}$ is defined as 
$$\normn{X} = \sum_{i=1}^{\min{\{d_1,d_2\}}}{\sigma_i\left(X\right)}$$
 Where $\sigma_i\left(X\right)$ is the $i^{th}$ largest singular value of $X$ given by the singular value decomposition $X = \sum_{i=1}^{\min{\{d_1,d_2\}}}{\sigma_i\left(X\right) u_i v_i^T}$. With the trace inner product $\inner{X}{Y} = \trace\left(X^T Y\right)$, nuclear norm satisfies conditions \ref{condition-2} and \ref{condition-1}. In this case, $K(X)  = {\rm rank}(X)$, $\gamma_i= \sigma_{i}\left(X\right)$ and $a_i = u_i v_i^T$ for $i \in \{1, 2, \ldots, {\rm rank}(X)\}$. The subspace $T_X$ is given by:
 $$T_X = \left\{\left.\sum_{i=1}^{{\rm rank}(X)} u_i z_i^T + {z'}_i v_i^T\; \right| \; z_i \in \mathbb{R}^{d_2}, {z'}_i \in \mathbb{R}^{d_1}, 
 \; \text{for all}\; i\right\},$$
 \noindent while $e_X = \sum_{i=1}^{{\rm rank}(X)} u_i v_i^T$.

\item{\bf Weighted $ \ell_1$ norm} on $\mathbb{R}^{n}$ is defined as:
$$\norml{x} = \sum_{i=1}^{n}{w_i |x_i|}$$

\noindent where $w$ is a vector of positive weights. With $\inner{x}{y} = \sum_{i=1}^{n} w_i^2 x_i y_i$, $\ell_1$ norm satisfies conditions \ref{condition-2} and \ref{condition-1}. For $\ell_1$ norm, $K(x) = |\{i|x_i \neq 0 \}|$, $\{\gamma_1,\gamma_2, \ldots, \gamma_k\}= \{w_i |x_{i}|  \;|\;  |x_{i}| > 0, i = 1,\ldots,n\}$. $T_x$ is the support of $x$, which is defined as:
$$T_x = \{y \in \mathbb{R}^n \;|\; y_i = 0 \text{  if  } x_i = 0\},$$
\noindent while the $i^{th}$ element of $e_x$ is ${\rm sign}(x_i) w_i$.

\item{\bf $\ell_{1,2}$ norm on $\mathbb{R}^{d_1\times d_2}$:}   For a given inner product $\inner{\cdot}{\cdot}: \mathbb{R}^{d_1}\times \mathbb{R}^{d_1} \mapsto \mathbb{R}$ and its induced norm $\normu{\cdot}$ on $\mathbb{R}^{d_1}$, We define:
$$\norm{X}_{1,2} = \sum_{i=1}^{d_2}{\normu{X_i}},$$
\noindent  where $X_i$ denotes the $i^{th}$ column of $X$. With inner product  $\inner{X}{Y} = \sum_{i=1}^{d_2}{\inner{X_i}{Y_i}}$, $\ell_{1,2}$ norm satisfies conditions \ref{condition-2} and \ref{condition-1}. For this norm, $K(X) = |\{i|X_i \neq 0 \}|$ and $\{\gamma_1,\gamma_2,\ldots,\gamma_k\}= \{\normu{X_{i}}\;|\;\normu{X_{i}} > 0, i = 1, \ldots, d_2\}$. $T_X$ is the column support of $X$, which is defined as:
$$T_X = \left\{\left.[Y_1, Y_2 , \ldots, Y_{d_2}] \in \mathbb{R}^{d_1 \times d_2} \;\right|\; Y_i = 0 \text{  if  } X_i = 0\right\},$$
\noindent while the $i^{th}$ column of $e_X$ is equal to $0$ if $X_i = 0$ and is equal to ${X_i}/{\normu{X_i}}$ otherwise.

\end{itemize}

Our second result on properties of decomposable norms is captured in the next theorem which establishes sublinearity of $K$ for decomposable norms.

\begin{thm}\label{thm-con-2}
For all $x, y \in \mathbb{R}^n$ 
\begin{equation}
K\left(x+y\right) \leq K\left(x\right)+K\left(y\right) .
\end{equation}

\end{thm}

Theorem \ref{thm-con-2} for $\ell_1$, $\ell_{1,2}$ and nuclear norm is equivalent to sublinearity of cardinality of vectors, number of non-zero columns and rank of matrices. The proof of this theorem is included in Appendix B.

\subsection{Properties of $A$}

Restricted Isometry Property was first discussed in \cite{candes2006near} for sparse vectors. Generalization of that concept to low rank matrices was introduced in \cite{recht2010guaranteed}. Note that if $K\left(x\right) \leq k$, then $\norm{x} \leq \sqrt{k} \normu{x}$. Based on this observation we define restricted isometry constants of $A \in \mathbb{R}^{m\times n}$ as:

\begin{definition}
The upper (lower) restricted isometry constant $\rho_{+}\left({A},k\right)$ ($\rho_{-}\left({A},k\right)$) of 
a matrix $A\in \mathbb{R}^{m\times n}$ is the smallest (largest) positive constant that satisfies this inequality:
$$\rho_{-}\left(A,k\right) \normu{x}^2 \leq \normu{A x}^2 \leq \rho_{+}\left(A,k\right) \normu{x}^2,$$

\noindent whenever $ \norm{x}^2 \leq k \normu{x}^2 $. 
\end{definition}


\begin{proposition}\label{str-lip}
Let $A \in \mathbb{R}^{m\times n}$ and $f\left(x\right) = \frac{1}{2} \normu{A x-b}^2$. Suppose that $\rho_{+}\left(A,k\right)$ and $\rho_{-}\left(A,k\right)$ are restricted isometry constants corresponding to $A$, then:
\begin{equation}\label{str}
f\left(y\right) \geq f\left(x\right) + \langle \nabla f\left(x\right) , y-x \rangle + \frac{1}{2} \rho_{-}\left(A,k\right) \normu{x-y}^2,
\end{equation}
\begin{equation}\label{lip}
f\left(y\right) \leq f\left(x\right) + \langle \nabla f\left(x\right) , y-x \rangle + \frac{1}{2} \rho_{+}\left(A,k\right) \normu{x-y}^2,
\end{equation}
for all $x, y \in \mathbb{R}^{ n}$ such that $ \norm{x-y}^2 \leq k \normu{x-y}^2 $.
\end{proposition}

Proposition \eqref{str-lip} follows from the definition of restricted isometry constants and the following equality:
$$\frac{1}{2} \normu{A\left(x-y\right)}^2 = f\left(y\right) - f\left(x\right) - \langle \nabla f\left(x\right) , y-x \rangle.$$


%% file: Convergence.tex

\section{Proximal-gradient method and homotopy algorithm}
We state the proximal-gradient method and the  homotopy algorithm for the following optimization problem:
$${\operatorname{minimize}}{\quad \phi_{\lambda}\left(x\right) = f\left(x\right)+\lambda \norm{x}},$$

\noindent where $f\left(x\right) = \frac{1}{2} \normu{{A}x-b}^2$. While, for simplicity, we analyze the homotopy algorithm for the least squares loss function, the analysis can be extended to every function of form $f(x) = g(A x)$ when $g$ is a differentiable strongly convex function with Lipschitz continuous gradient.. 
The key element in the proximal-gradient method is the proximal operator which was developed by Moreau 
\cite{moreau1962fonctions} and later extended to maximal monotone operators by 
Rockafellar \cite{rockafellar1976monotone}. 
Nesterov has proposed several variants of the proximal-gradient methods \cite{nesterov2013gradient}.  In this section, we discuss the gradient method with adaptive line search. For any $x, y \in \mathbb{R}^n$ and positive $L$, we define: 
\begin{align*}
m_{\lambda, L}\left(y,x\right) &= f\left(y\right) + \langle \nabla f\left(y\right) , x-y \rangle + \frac{L}{2} \normu{x-y}^2 +\lambda \norm{x},\\
{\rm Prox}_{\lambda,L}\left(y\right) &= \argmin_{x \in \mathbb{R}^n}{m_{\lambda, L}\left(y,x\right)} \\
\omega_{\lambda}\left(x\right) &= \min_{\xi \in \partial  \norm{x}}{\dnorm{\lambda \xi + \nabla f\left(x\right)}}.
\end{align*}

Xiao and Zhang \cite{xiao2013proximal} have considered the proximal-gradient homotopy algorithm for $\ell_1$ norm. Here we state it for general norms.
Algorithm \eqref{homotopy}, introduces the homotopy algorithm and contains the proximal-gradient method as a subroutine. The stopping criteria in the proximal-gradient method is based on the quantity $$\dnorm{M_{t} (x^{(t-1)} - x^{(t)}) +\nabla f(x^{(t)})- \nabla f(x^{(t-1)})},$$ which is an upper bound on $\omega_{\lambda}(x^{(t)})$. This follows from the fact that since $x^{(t)} = \argmin_{x \in \mathbb{R}^n}{m_{\lambda, M_{t}}(x^{(t-1)}, x)}$, there exists $\xi \in \partial \norm{x^{(t)}}$ such that $\nabla f(x^{(t-1)}) + \lambda \xi + M_{t} (x^{(t)} - x^{(t-1)})=0$. Therefore,
\begin{align}\notag
    \omega_{\lambda}(x^{(t)}) \leq \dnorm{\lambda \xi + \nabla f(x^{(t)})} 
    &=  \dnorm{\lambda \xi + \nabla f(x^{(t-1)}) +\nabla f(x^{(t)})- \nabla f(x^{(t-1)})}\\   \label{omega-1}
    &\leq  \dnorm{M_{t} (x^{(t-1)} - x^{(t)}) +\nabla f(x^{(t)})- \nabla f(x^{(t-1)})}.
    \end{align}


 
The homotopy algorithm
reduces the value of $\lambda$ in a series of steps and in each step applies the proximal-gradient method.
At step $t$, $\lambda_t = \lambda_0 \eta^{t}$ and $\epsilon_t = \delta' \lambda_t$ with $\eta \in (0,1)$ and $\delta' \in \left(0,1\right)$. In the proximal-gradient method and the backtracking subroutine, the parameters
\ $\gamma_{\rm{dec}}\geq 1$ and $\gamma_{inc
} > 1$ should be initialized. Since the function $f$ satisfies the inequality \eqref{lip1}, it is clear that $L_{\min}$ should be chosen less than $L_f$.

\begin{algorithm}
\renewcommand{\algorithmicrequire}{\textbf{Input:}}
 \renewcommand{\algorithmicensure}{\textbf{Parameters:}}
\caption{Homotopy}
\label{homotopy}
\begin{algorithmic}
\Require{$ \lambda_{\rm tgt} > 0,\, \epsilon > 0$}
\Ensure{$\eta \in \left(0,1\right), \, \delta' \in \left(0,1\right), \,  L_{\min} > 0$}
\State{$y^{\left(0\right)} \leftarrow 0,\,\lambda_0 \leftarrow \dnorm{A^* b}, \, M \leftarrow L_{\min}, \, N \leftarrow \lfloor \log\left(\frac{\lambda_{\rm \rm tgt}}{\lambda_0}\right)/\log\left(\eta\right)\rfloor$}
\For{$t = 0,1,\ldots, N-1$}
\State{$\lambda_{t+1} \leftarrow \eta \lambda_{t} $}
\State{$\epsilon_t \leftarrow \delta' \lambda_t$}
\State {$[y^{(t+1)}, M] \leftarrow $ ProxGrad\_${\phi_{\lambda_{t+1}}}$  $\left(y^{(t)},M,L_{\min},\epsilon_t\right)$}
\EndFor
 \renewcommand{\algorithmicensure}{\textbf{Parameter:}}
\State{$[y ,M] \leftarrow $ ProxGrad\_${\phi_{\lambda_{\rm \rm tgt}}}$  $\left(y^{(N)},M,L_{\min},\epsilon\right)$}
\Algphase{{\bf Subroutine 1} $[x,  M] = $ ProxGrad\_${\phi_{\lambda}}$  $\left(x^{(0)},L_0,L_{\min},\epsilon'\right)$}
\Ensure{$\gamma_{\rm dec}\geq 1, $}
\State{$t \leftarrow 0$}
\Repeat{}
\State $[x^{\left(t+1\right)}, M_{t+1}]\leftarrow \text{Backtrack\_}_{\phi_{\lambda}}\left(x^{\left(t\right)},L_{t}\right)$
\State{$L_{t+1}\quad\quad\quad \leftarrow \max\{L_{\min}, M_{t+1}/\gamma_{\rm{dec}}\}$}
\State{$t \leftarrow t+1$}
\Until{{$\dnorm{M_{t}(x^{(t-1)} - x^{(t)}) + \nabla f\left(x^{\left(t\right)}\right) - \nabla f\left(x^{\left(t-1\right)}\right)}\leq \epsilon'$}}
\State{$x \leftarrow x^{\left(t\right)}, \, M \leftarrow M_t$}
\Algphase{{\bf Subroutine 2} $[y , M] = $ Backtrack\_$\phi_{\lambda}$ $\left(x,L\right)$}
\Ensure{$ \gamma_{\rm inc} > 1$}
\While{$\phi_{\lambda}\left({\rm Prox}_{\lambda,L}\left(x\right)\right)> m_{\lambda, L}\left(x,{\rm Prox}_{\lambda,L}\left(x\right)\right)$}
\State {$\quad  \quad \quad L\leftarrow {\gamma_{\rm{inc}}}L$} 
\EndWhile
\State{$y \leftarrow {\rm Prox}_{\lambda,L}\left(x\right), \, M \leftarrow L$}
\end{algorithmic}
\end{algorithm}
 Theorem 5 in \cite{nesterov2013gradient} states that the proximal-gradient method has a linear rate of convergence when $f$ satisfies \eqref{str1} and \eqref{lip1}. In proposition~\ref{convresult} we restate that theorem with minimal assumptions which is $f$ satisfies \eqref{str1} and \eqref{lip1} on a restricted set. The proof of this proposition is given in appendix B.

\begin{proposition}\label{convresult}
Let $x^* \in \argmin \phi_\lambda$. If for every $t$:
\begin{align}\label{str1-itr}
f\left(x^{\left(t\right)}\right) &\geq f\left(x^*\right) + \langle \nabla f\left(x^*\right) , x^{\left(t\right)}-x^* \rangle + \frac{\mu_f}{2} \normu{x^{\left(t\right)}-x^*}^2,\\ \label{str2-itr}
f\left(x^{\left(t+1\right)}\right) &\geq f\left(x^{\left(t\right)}\right) + \langle \nabla f\left(x^{\left(t\right)}\right) ,x^{\left(t+1\right)}-x^{\left(t\right)} \rangle + \frac{\mu_f}{2} \normu{x^{\left(t\right)}-x^{\left(t+1\right)}}^2,\\ \label{lip1-itr}
f\left(x^{\left(t+1\right)}\right) &\leq f\left(x^{\left(t\right)}\right) + \langle \nabla f\left(x^{\left(t\right)}\right) , x^{\left(t+1\right)}-x^{\left(t\right)} \rangle + \frac{L_f}{2} \normu{x^{\left(t\right)}-x^{\left(t+1\right)}}^2,
\end{align}
\noindent then
\begin{equation}\label{first-co}
\phi_{\lambda}\left(x^{\left(t\right)}\right)-\phi_{\lambda}\left(x^*\right) \leq \left(1-\frac{\mu_f \gamma_{\rm{inc}}}{4 L_f}\right)^t \left(\phi_{\lambda}\left(x^{(0)}\right)-\phi_{\lambda}\left(x^*\right)\right).
\end{equation}
In addition, if
\begin{equation}\label{lip2-itr}
\dnorm{\nabla f\left(x^{\left(t\right)}\right) - \nabla f\left(x^{\left(t+1\right)}\right)} \leq L'_f \normu{x^{\left(t\right)} - x^{\left(t+1\right)}}
\end{equation}
\noindent and 
\begin{equation}\label{dnorm-normu} 
\dnorm{x^{\left(t\right)} - x^{\left(t+1\right)}} \leq \theta \normu{x^{\left(t\right)} - x^{\left(t+1\right)}}
\end{equation}
\noindent for some constants $\theta$ and $L'_f$, then
\begin{align}\notag
\omega_{\lambda}\left(x^{\left(t+1\right)}\right) &\leq \dnorm{M_{t+1}(x^{(t)} - x^{(t+1)}) + \nabla f\left(x^{\left(t+1\right)}\right) - \nabla f\left(x^{\left(t\right)}\right)}\\ \label{second-co}
 &\leq \theta \left(1+\frac{L'_f}{\mu_f}\right)\sqrt{ 2 \gamma_{\rm{inc}}L_f \left(\phi_{\lambda}\left(x^{\left(t\right)}\right)-\phi_{\lambda}\left(x^*\right)\right) }.
\end{align}
    \end{proposition}


\section{Convergence result}\label{sec::conv}

First note that since the objective function is not strongly convex if one applies the sublinear convergence rate of proximal gradient method, the iteration complexity of the homotopy algorithm is $O(\frac{1}{\epsilon} + \sum_{t=1}^{N}\frac{1}{\delta' \lambda_t})$ which can be simplified to $O(\frac{1}{\epsilon} + \frac{1}{\delta' (1-\eta) \lambda_{tgt}})$.
As it was stated in the introduction, we use the structure of this problem to provide a linear rate of convergence when assumptions similar to those needed to derive recovery bounds hold. 

Suppose $b = A x_0 + z$, for some $x_0 \in \mathbb{R}^n$ and $z \in \mathbb{R}^m$.  Here, $z$ is the noise vector that is added to linear measurements from an structured model $x_0$. Also, we define $k_0 := K(x_0)$ and the constant $c$: 
$$c := \max_{x \in T_{x_0}-\{0\}} \frac{\norm{x}^2}{k_0 \normu{x}^2}.$$ 

Note that {$c = 1$ for $\ell_1$ and $\ell_{1,2}$ norms, and $c \leq 2$ for nuclear norm}. This follows from the fact that $K(x) = k_0$ when $x \in T_{x_0}$ for $\ell_1$, $\ell_{1,2}$ norms, while $K(x) \leq 2 k_0$ when $x \in T_{x_0}$ in case of nuclear norm. Through out this section, we assume the regularizing norm satisfies conditions \ref{condition-2} and \ref{condition-1} introduced in Section \ref{sec:properties}. Before we state the convergence theorem, we introduce an assumption:

\begin{assumption}\label{assumption-2}
$\lambda_{\rm tgt}$ is such that $\dnorm{A^* z} \leq \frac{\lambda_{\rm tgt}}{4}$. Furthermore, there exist constants $r > 1$ and $\delta \in  (0,\; \frac{1}{4}]$ such that:
\begin{align}\label{rho}
&\frac{\rho_{-}\left(A,{ c k_0\left(1+\gamma\right)^2 }\right)}{\rho_{+}\left(A,72 r c k_0 (1+\gamma) \gamma_{\rm inc}\right)}> \frac{ c}{ r}\\ \label{rho2}
& \rho_{-}\left(A,72 r c k_0 (1+\gamma) \gamma_{\rm inc}\right) > 0
\end{align}
\noindent where:
\begin{align}
&\gamma := \frac{\lambda_{\rm tgt}\left(1+ \delta\right) +\dnorm{A^* z}}{\lambda_{\rm tgt}\left(1-\delta\right) -\dnorm{A^* z}}.
\end{align}

\end{assumption}

We define $\tilde{k} = 36 r c k_0 (1+\gamma) \gamma_{\rm inc}$.  In appendix A, we provide an upper bound on the number of measurement needed for \eqref{rho} to be satisfied with high probability whenever rows of $A$ are sub-Gaussian random vectors. 

The next theorem establishes the linear convergence of the proximal gradient method when $\omega_{\lambda}\left(x^{\left(0\right)}\right) = \min_{\xi \in \partial \norm{x^{\left(0\right)}}}{\dnorm{\nabla f\left(x\right) + \lambda \xi}}$ is sufficiently small, while Theorem \ref{convresult3} establishes the overall linear rate of convergence of homotopy algorithm.

 \begin{thm}\label{convresult2}
  Let $x^{\left(t\right)}$ denote the $t^{\text{th}}$ iterate of ProxGrad\_${\phi_{\lambda}}$  $\left(x^{(0)},L_0,L_{\min},\epsilon'\right)$, and let $x^* \in \argmin{\phi_{\lambda}\left(x\right)}$. Suppose Assumption \ref{assumption-2} holds true for some $r$ and $\delta$, $L_{\min} \leq \gamma_{\rm inc} \rho_{+}\left(A,2 \tilde{k}\right)$, and $\lambda \geq \lambda_{\rm tgt}$. If $x^{\left(0\right)}$ satisfies:
    $$K\left(x^{\left(0\right)}\right)\leq \tilde{k}, \quad \omega_{\lambda}\left(x^{\left(0\right)}\right) \leq \delta \lambda,$$
\noindent then:
 \begin{equation}
K\left(x^{\left(t\right)}\right)\leq \tilde{k},
\end{equation}
\begin{equation}\label{2-first-co}
\phi_{\lambda}\left(x^{\left(t\right)}\right)-\phi_{\lambda}\left(x^*\right) \leq \left(1-\frac{1}{4\gamma_{\rm inc}\kappa}\right)^t \left(\phi_{\lambda}\left(x^{\left(0\right)}\right)-\phi_{\lambda}\left(x^*\right)\right),
\end{equation}
 \noindent and
    \begin{equation}\label{2-second-co}
\omega_{\lambda}\left(x^{\left(t\right)}\right) \leq \left(1+\frac{\sqrt{\rho_+\left(A,1\right)\rho_+\left(A,2\tilde{k}\right)}}{\rho_-\left(A,2\tilde{k}\right)}\right)\sqrt{ 2 \gamma_{\rm inc}\rho_+\left(A,2\tilde{k}\right) \left(\phi_{\lambda}\left(x^{\left(t-1\right)}\right)-\phi_{\lambda}\left(x^*\right)\right)},
    \end{equation} 
\noindent where $\kappa = \frac{\rho_+\left(A,2\tilde{k}\right)}{\rho_-\left(A,2\tilde{k}\right)}.$    
    \end{thm}

\begin{thm}\label{convresult3}
Let $y^{\left(t\right)}$ denote the $t^{th}$ iterate of Homotopy algorithm, and let $y^* \in \argmin \phi_{\lambda_{\rm \rm tgt}}\left(y\right)$. Suppose Assumption \ref{assumption-2} holds true for some $r$ and $\delta$, $L_{\min} \leq \gamma_{\rm inc} \rho_{+}\left(A,2 \tilde{k}\right)$, and $\lambda_0 \geq \lambda_{tgt}$. Furthermore, suppose that $\delta'$ and $\eta$ in the algorithm satisfy:
\begin{align}\label{eta}
\frac{1+\delta'}{1+\delta}\leq \eta .
\end{align}




When $t =0, 1,\ldots,  N-1 $, the number of proximal-gradient iterations for computing $y^{\left(t\right)}$  is bounded by  
\begin{align}\label{inner-iteration-bound-1}
\frac{\log\left(C/\delta^2\right)}{\log\left(1-\frac{1}{4\gamma_{\rm inc}\kappa}\right)^{-1}},
\end{align}
%

\noindent The number of proximal-gradient iterations for computing $y$ is bounded by  
\begin{align}\label{inner-iteration-bound-2}
\frac{\log\left(C \lambda_{\rm \rm tgt}/\epsilon^2  \right)}{ \log\left(1-\frac{1}{4\gamma_{\rm inc}\kappa}\right)^{-1}},
\end{align}
where $C:= {6 \gamma_{\rm inc} \kappa \delta ck_{0}\left(1+\gamma\right) \left(\sqrt{\rho_{-}\left(A,2 \tilde{k}\right)}+{\sqrt{\rho_{+}\left(A,1\right) \kappa}}\right)^2}\Bigg/{\rho_{-}\left(A,{ c\left(1+\gamma\right)^2 k_{0}}\right)}$ and $\kappa = \frac{\rho_+\left(A,2\tilde{k}\right)}{\rho_-\left(A,2\tilde{k}\right)}$. The objective gap of the output $y$ is bounded by
$$\phi_{\lambda_{\rm \rm tgt}}\left(y\right)-\phi_{\lambda_{\rm \rm tgt}}\left(y^*\right) \leq \frac{9  ck_{0} \lambda_{\rm \rm tgt}\left(1+\gamma\right) \epsilon}{ \rho_{-}\left(A,{ c\left(1+\gamma\right)^2 k_{0}}\right)},$$
\noindent while the total number of iterations for computing $y$ is bounded by:
\begin{align*}
\frac{\log\left(C \lambda_{\rm \rm tgt}/\epsilon^2  \right) +(  \log\left(\frac{\lambda_{\rm \rm tgt}}{\lambda_0}\right)/\log\left(\eta\right) ) \log\left(C/\delta^2\right)}{ \log\left(1-\frac{1}{4\gamma_{\rm inc}\kappa}\right)^{-1}}.
\end{align*}

\end{thm}

\subsection{Parameters selection satisfying the assumptions}
Four parameters of $L_{\min}$, $\lambda_{\rm tgt}$, $\delta'$ and $\eta$ should be set in the homotopy algorithm. The assumption on $L_{\min}$ is only for convenience. If $L_{\min} > \gamma_{\rm inc} \rho_{+}\left(A,2 \tilde{k}\right)$, one can replace $\gamma_{\rm inc} \rho_{+}\left(A,2 \tilde{k}\right)$ with $L_{\min}$ in the analysis.

Assumption \ref{assumption-2} requires $\lambda_{\rm tgt} \geq 4 \dnorm{A^* z}$. This assumption on the regularization parameter is a standard assumption that is used in the literature to provide optimal bounds for recovery error \cite{candes2011tight,candes2007dantzig,negahban2012unified}. The lower bound on $\lambda_{\rm tgt}$, ensures $\gamma \leq \frac{5+4\delta}{3-4\delta}$. If we choose $\delta$ and $\eta$, we can set $\delta' = (1+\delta)\eta - 1$ to ensure that it satisfies \eqref{eta}. 
The parameter $\delta$ is directly related to satisfiability of \eqref{rho} in Assumption \ref{assumption-2}. For example, if $\delta = 1/12$, then $\gamma \leq 2$ and Assumption \ref{assumption-2} is satisfied with $r = 2 c$ if:
\begin{align*}
&\frac{\rho_{-}\left(A,{ 9 c k_0 }\right)}{\rho_{+}\left(A,432 c^2 k_0  \gamma_{\rm inc}\right)}> \frac{ 1}{ 2},\\
& \rho_{-}\left(A,432 c^2 k_0 \gamma_{\rm inc}\right) > 0.
\end{align*}

Theoretically, the optimal choice of $\delta$ maximizes $\kappa$ subject to existence of $r>1$ that satisfies \eqref{rho} and \eqref{rho2}.
%
 In appendix A, we provide an upper bound on the number of measurement needed for \eqref{rho} and \eqref{rho2} to be satisfied with high probability for given $\delta$ and $r>1$ whenever rows of $A$ are sub-Gaussian random vectors. The parameter $\eta$ should be chosen to be greater than $\frac{1}{2}$ for \eqref{eta} to be satisfied.
 
 \subsection{Convergence proof}
The main part of the proof of Theorems \ref{convresult2} and \ref{convresult3} is establishing the fact that $K\left(x^{\left(t\right)}\right)\leq \tilde{k}$. Given that $K\left(x^{\left(t\right)}\right)\leq \tilde{k}$ for all $t$, Proposition \ref{str-lip} ensures that hypothesis of  Proposition \ref{convresult}, i.e., strong convexity and gradient Lipschitz continuity over a restricted set, are satisfied.
We adapt the same strategy as in \cite{xiao2013proximal} and prove that $K\left(x^{\left(t\right)}\right)\leq \tilde{k}$ in a series of three lemmas. We have written the statement of the lemmas here, while their proofs are given in Appendix B.  Lemma~\ref{conv-lemma-1}  states that if $\omega_{\lambda}(x)$ does not exceed a small fraction of $\lambda$, then $x$ is close to $x_0$.  

\begin{lemma}\label{conv-lemma-1}
If $\omega_{\lambda}(x) \leq \delta \lambda$ and $\rho_{-}\left(A,{ c\left(1+\gamma\right)^2 k_{0}}\right) > 0$, then:
%
%
\begin{align}
\max{\{\norm{x-x_{0}},\frac{1}{\delta \lambda} \left(\phi_{\lambda}\left(x\right)-\phi_{\lambda}\left(x_0\right)\right)\}} \leq \frac{{ck_{0}\left(1+\gamma\right)\left(\left(1+\delta\right)\lambda + \dnorm{A^* z}\right)}}{\rho_{-}\left(A,{ c\left(1+\gamma\right)^2 k_{0}}\right)}.
\end{align}
\end{lemma}

Note that if $\lambda \geq 4 \dnorm{A^* z}$ and $\delta \leq \frac{1}{4}$, we can simplify the conclusion of Lemma \ref{conv-lemma-1} as
$$\max{\{\norm{x-x_{0}},\frac{1}{\delta \lambda} \left(\phi_{\lambda}\left(x\right)-\phi_{\lambda}\left(x_0\right)\right)\}} \leq
  \frac{3 ck_{0}\lambda\left(1+\gamma\right)}{2 \rho_{-}\left(A,{ c\left(1+\gamma\right)^2 k_{0}}\right)}$$

While the hypotheses of this lemma is true in the first step of every outer iteration of homotopy algorithm,  $\omega_{\lambda}(x^{(t)})$ may not be decreasing in proximal-gradient algorithm. However, the objective decreases after every iteration of the proximal-gradient algorithm. Thus to conclude that $x^{(t)}$ is close to $x_0$ in all the inner proximal-gradient steps we can use the following lemma:
\begin{lemma}\label{conv-lemma-2}
Suppose Assumption \ref{assumption-2} holds true, and $\lambda \geq \lambda_{\rm tgt}$. If 
$$\phi_{\lambda}\left(x\right)-\phi_{\lambda}\left(x_0\right)  \leq  \frac{3 ck_{0}\delta \lambda^2\left(1+\gamma\right)}{2 \rho_{-}\left(A,{ c\left(1+\gamma\right)^2 k_{0}}\right)},$$
 then
$$\max{\{\frac{1}{2 \lambda}\normu{A\left(x-x_{0}\right)}^2,\norm{x-x_{0}}\}} \leq  \frac{9 ck_{0} \lambda\left(1+\gamma\right)}{2 \rho_{-}\left(A,{ c\left(1+\gamma\right)^2 k_{0}}\right)}.$$

\end{lemma}

The proofs of Lemma~\ref{conv-lemma-1} and Lemma~\ref{conv-lemma-2} generalize the proofs of the corresponding lemmas in \cite{xiao2013proximal} given for $\ell_1$ norm to norms that satisfy Condition \ref{condition-1} using the structure of $\partial{\norm{x_0}}$ given by \eqref{decomposable}. The last lemma provides an upper bound on $K\left(x^{+}\right)$, where $x^{+}$ is produced via a proximal-gradient step on $x$, as long as $x$ satisfies the conclusion of Lemma~\ref{conv-lemma-2} and Assumption~\ref{assumption-2} holds. The proof of Lemma~\ref{conv-lemma-3} uses a slightly different approach than the one given in \cite{xiao2013proximal} resulting in a simpler requirement on $\tilde{k}$ in Assumption~\ref{assumption-2}. 

\begin{lemma}\label{conv-lemma-3} Let $x^{+} = {\rm Prox}_{\lambda,L}\left(x\right)$ and suppose Assumption \ref{assumption-2} holds true, and $\lambda \geq \lambda_{\rm tgt}$. If $L \leq \gamma_{\rm inc} \rho_{+}\left(A,2\tilde{k}\right)$ and $$\max{\{\frac{1}{2 \lambda}\normu{A\left(x-x_{0}\right)}^2,\norm{x-x_{0}}\}} \leq  \frac{9 ck_{0} \lambda\left(1+\gamma\right)}{2 \rho_{-}\left(A,{ c\left(1+\gamma\right)^2 k_{0}}\right)},$$ then $K\left(x^{+}\right) \leq \tilde{k}$.
\end{lemma}

\subsection{Proof of Theorem \ref{convresult2}}
First we show that $L_t \leq \gamma_{\rm inc} \rho_{+}\left(A, 2\tilde{k}\right)$ and $K\left(x^{\left(t\right)}\right) \leq \tilde{k}$ for all $t\geq 0$. The inequalities hold true for $t=0$ by the hypothesis. Suppose $L_t \leq \gamma_{\rm inc} \rho_{+}\left(A, \tilde{k}\right)$ and $K\left(x^{\left(t\right)}\right) \leq \tilde{k}$ for some $t \geq 0$. Since $\phi_{\lambda}{\left(x^{\left(t\right)}\right)} \leq \phi_{\lambda}{\left(x^{\left(0\right)}\right)}$, by Lemma \ref{conv-lemma-2}, we have:
$$\max{\{\frac{1}{2 \lambda}\normu{A\left(x^{(t)}-x_{0}\right)}^2,\norm{x^{(t)}-x_{0}}\}} \leq  \frac{9 ck_{0} \lambda\left(1+\gamma\right)}{2 \rho_{-}\left(A,{ c\left(1+\gamma\right)^2 k_{0}}\right)}.$$

By Lemma \ref{conv-lemma-2}, Lemma \ref{conv-lemma-3} and Theorem \ref{thm-con-2}, for any $L \leq \gamma_{\rm inc} \rho_{+}\left(A, 2\tilde{k}\right)$
\begin{align*}
&K\left({\rm Prox}_{\lambda,L}\left(x^{\left(t\right)}\right)\right) \leq \tilde{k}, \\
 &K\left({\rm Prox}_{\lambda,L}\left(x^{\left(t\right)}\right) - x^{\left(t\right)}\right) \leq 2 \tilde{k}.
\end{align*}

Now we can use Proposition \ref{str-lip} to conclude that $M_{t+1} \leq  \gamma_{\rm inc}\rho_{+}\left(A, 2\tilde{k}\right)$ hence  $L_{t+1} \leq M_{t+1} / \gamma_{\rm dec} \leq  \gamma_{\rm inc} \rho_{+}\left(A, 2\tilde{k}\right)$. In addition, by Lemma \ref{conv-lemma-3}, $K\left(x^{\left(t+1\right)}\right) = K\left({\rm Prox}_{\lambda,M_{t+1}}\left(x^{\left(t\right)}\right)\right) \leq \tilde{k}$. 

Since ${\rm Prox}_{\lambda,L}{\left(x^*\right)} = x^*$ for any $L>0$, by Lemmas \ref{conv-lemma-1}, \ref{conv-lemma-2}, and \ref{conv-lemma-3}, $K\left(x^*\right)\leq \tilde{k}$. By Theorem \ref{thm-con-2}, we have:
$$K\left(x^{\left(t+1\right)} - x^{\left(t\right)}\right) \leq 2 \tilde{k}, \; K\left(x^{\left(t\right)} - x^*\right) \leq 2 \tilde{k},$$
\noindent which yields
\begin{align}\notag
\dnorm{A^* A \left(x^{\left(t+1\right)} - x^{\left(t\right)}\right)} &= \max_{a \in \mathcal{G}_{\norm{\cdot}}}{\inner{a}{A^* A \left(x^{\left(t+1\right)} - x^{\left(t\right)}\right)}}\\ \label{L-prime}
&= \max_{a \in \mathcal{G}_{\norm{\cdot}}}{\inner{A a}{ A \left(x^{\left(t+1\right)} - x^{\left(t\right)}\right)}} \leq \sqrt{\rho_+\left(A,1\right)\rho_+\left(A,2\tilde{k}\right)} \normu{x^{\left(t+1\right)}-x^{\left(t\right)}}.
\end{align}

Now Proposition \ref{str-lip} and \eqref{L-prime} ensure that all the hypotheses of Proposition \ref{convresult} are satisfied with $\mu_f = \rho_-\left(A,2\tilde{k}\right)$, $L_f = \rho_+\left(A,2\tilde{k}\right)$. $L'_f = \sqrt{\rho_+\left(A,1\right)\rho_+\left(A,2\tilde{k}\right)}$ and $\theta = 1$. Thus the conclusion follows from Proposition \ref{convresult}.


\subsection{Proof of Theorem \ref{convresult3}}
Let $y^*_{t} \in \argmin \phi_{\lambda_{t}}\left(y\right)$ .
For the ease of notation let $\lambda_{N+1} \leftarrow \lambda_{\rm tgt}$. First we show that $\omega_{\lambda_{t+1}}\left(y^{\left(t\right)}\right) \leq \delta \lambda_{t+1}$ and $K\left(y^{\left(t\right)}\right) \leq \tilde{k}$ for $t = 0,1,\ldots, N$. When $t=0$, we have $y^{(0)} = 0$ and $\lambda_0 = \dnorm{A^* b}$. Therefore, $K(y^{(0)}) = 0$ and
\begin{align*}
\omega_{\lambda_{1}}(y^{(0)}) &= \min_{\xi \in \partial\norm{0}}\dnorm{A^*b + \lambda_1 \xi} \\
&\text{Since $\frac{-A^* b}{\lambda_0} \in \partial \norm{0} $}\\
&\leq \dnorm{A^*b - \frac{\lambda_1 }{\lambda_0} A^* b} \\
& = (1-\eta) \lambda_{0} \leq \delta \lambda_1,
\end{align*}
 where in the last inequality we used \eqref{eta}. Suppose $\omega_{\lambda_{t}}\left(y^{\left(t-1\right)}\right) \leq \delta \lambda_{t}$ and $K\left(y^{\left(t-1\right)}\right) \leq \tilde{k}$. By Theorem \ref{convresult2}, we have:
$$K\left(y^{\left(t\right)}\right) \leq \tilde{k}.$$

By \eqref{omega-1}, the stopping condition in the proximal gradient algorithm ensures $\omega_{\lambda_{t}}\left(y^{\left(t\right)}\right) \leq \delta' \lambda_{t}$. Therefore, there exists $\xi \in \partial \norm{y^{(t)}}$ such that $\dnorm{A^* \left(A y^{\left(t\right)} - b\right) + \lambda_t \xi} \leq \delta' \lambda_{t}$. Now using hypothesis \eqref{eta}, we get:
\begin{align*}
\omega_{\lambda_{t+1}}\left(y^{\left(t\right)}\right) &\leq \dnorm{A^* \left(A y^{\left(t\right)} - b\right) + \lambda_{t+1} \xi}\\
& \leq \dnorm{A^* \left(A y^{\left(t\right)} - b\right) + \lambda_{t} \xi } +  \dnorm{ \left(\lambda_{t+1} - \lambda_{t}\right)\xi}\\
& \leq \omega_{\lambda_{t}}{\left(y^{(t)}\right)} + \left(\lambda_{t} - \lambda_{t+1}\right) \leq     \left(-1+\left(\delta'+1\right)/\eta\right)\lambda_{t+1} \leq \delta \lambda_{t+1}.
\end{align*}

By Lemma \ref{conv-lemma-1} and the comment that follows it, for all $t = 0, \ldots, N $, we have
  \begin{align*}
 \norm{y^{\left(t\right)}-y^*_{t+1}} &\leq \norm{y^{\left(t\right)}-x_{0}} + \norm{y^*_{t+1}-x_{0}}\\
  &\leq \frac{{ck_{0}\left(1+\gamma\right)\left(\left(2+\delta\right)\lambda_{t+1} + 2\dnorm{A^* z}\right)}}{\rho_{-}\left(A,{ c\left(1+\gamma\right)^2 k_{0}}\right)} \\
  &\leq \frac{{3 ck_{0}\left(1+\gamma\right)\lambda_{t+1} }}{\rho_{-}\left(A,{ c\left(1+\gamma\right)^2 k_{0}}\right)}. 
\end{align*}

Hence
\begin{align*}
\phi_{\lambda_{t+1}}\left(y^{\left(t\right)}\right) -\phi_{\lambda_{t+1}}\left(y^*_{t+1}\right)&\leq \inner{\omega_{\lambda_{t+1}}\left(y^{(t)}\right)}{y^{(t)}-y^*_{t+1}}\\
&\leq {\omega_{\lambda_{t+1}}\left(y^{\left(t\right)}\right)}\norm{y^{\left(t\right)}-y^*_{t+1}}\\
&\leq \frac{{3 \delta ck_{0}\left(1+\gamma\right)\lambda_{t+1}^2 }}{\rho_{-}\left(A,{ c\left(1+\gamma\right)^2 k_{0}}\right)}.
\end{align*}

Now the upper bounds in \eqref{inner-iteration-bound-1} and \eqref{inner-iteration-bound-2} on the number of inner iterations follow from the second conclusion in Theorem \ref{convresult2}.

By \eqref{H-second}, we have
$$\norm{y-y^*} \leq  \norm{y-y_0} + \norm{y_0-y^*} \leq  \frac{9 ck_{0} \lambda_{\rm \rm tgt}\left(1+\gamma\right)}{ \rho_{-}\left(A,{ c\left(1+\gamma\right)^2 k_{0}}\right)}.$$

By convexity of $\phi_{\lambda_{\rm \rm tgt}}$, we get:
\begin{align*}
\phi_{\lambda_{\rm \rm tgt}}\left(y\right) -\phi_{\lambda_{\rm \rm tgt}}\left(y^*\right) &\leq  \inner{\omega_{\lambda_{tgt}}\left(y\right)}{y-y^*}\\
&\leq  \frac{9  ck_{0} \lambda_{\rm \rm tgt}\left(1+\gamma\right) \epsilon}{ \rho_{-}\left(A,{ c\left(1+\gamma\right)^2 k_{0}}\right)}.
\end{align*}

%% file: Numerical.tex
\begin{figure}
\begin{subfigure}{0.5\linewidth}
\centering
\includegraphics[width=.85\textwidth]{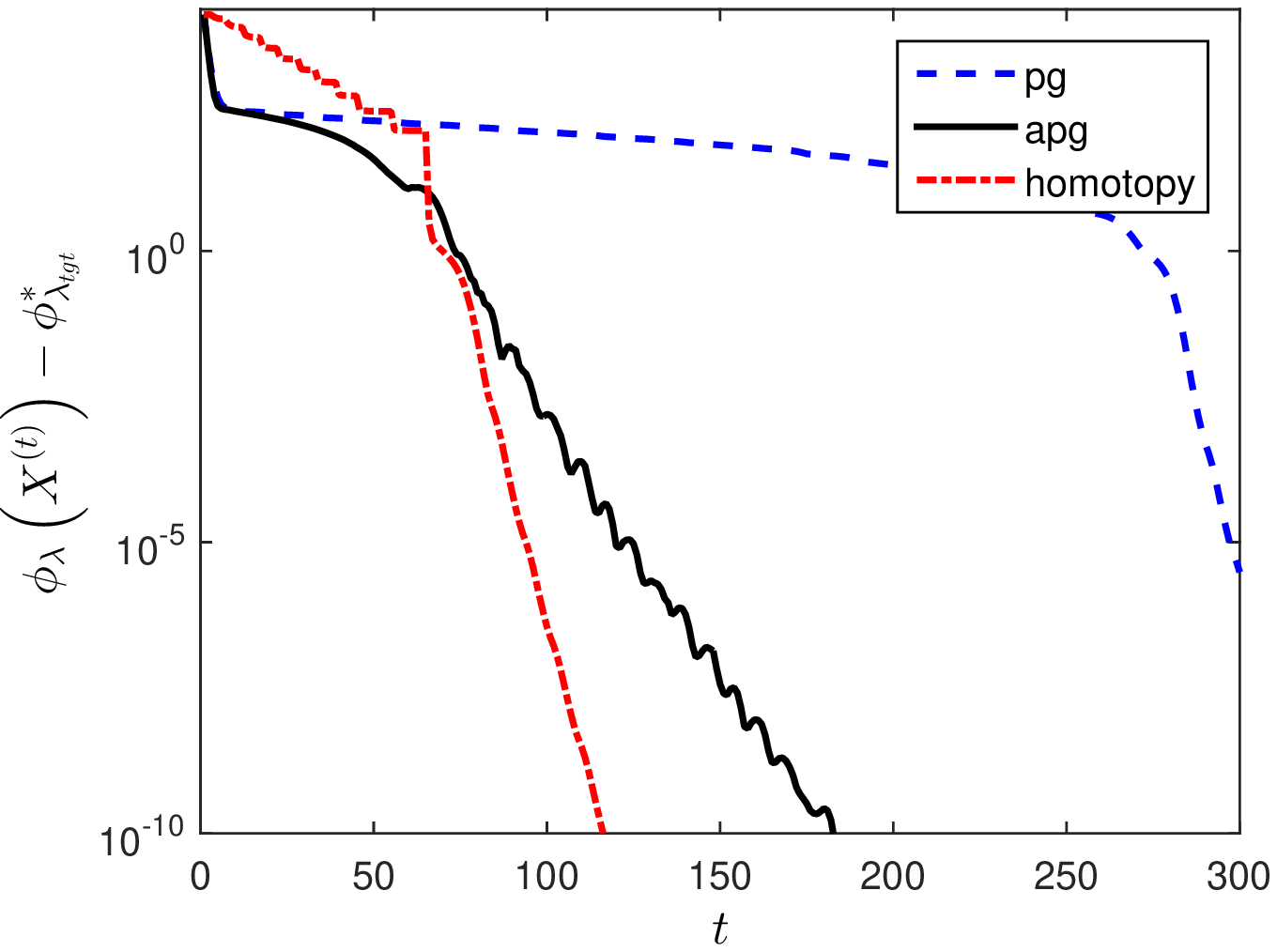}
\caption{Objective gap vs. iteration}\label{fig:duet}
\end{subfigure}%
\begin{subfigure}{0.5\linewidth}
\centering
\includegraphics[width=.85\textwidth]{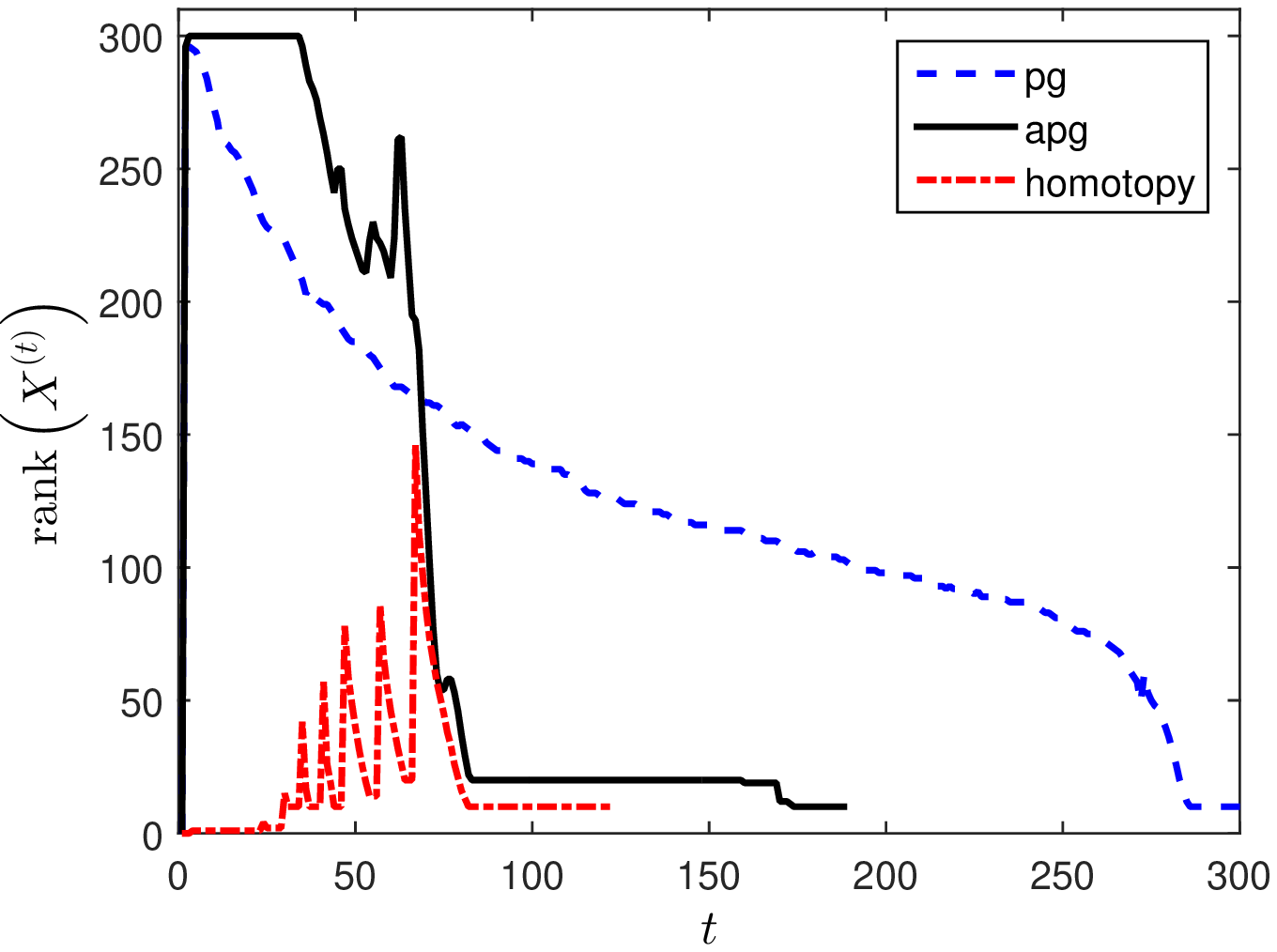}
\caption{Rank vs. iteration}\label{fig:duet2}
\end{subfigure}\\[1ex]
\caption{Comparison of homotopy, proximal-gradient and accelerated proximal-gradient algorithms for problem 1}\label{fig:results2}
\end{figure}

\section{Numerical Experiments}
We consider two problems. The details of each problem are summarized in the following table:

\begin{center}
\begin{tabular}{c |c| c|}
\cline{2-3}
&Problem $1$ & Problem $2$ \\
\cline{1-3} \cline{1-3}
\multicolumn{1}{ |c|| } {Objective} & $\frac{1}{2} \normu{A \vect(X)+b}^2+ \lambda \normn{X}$&$\frac{1}{2} \normu{ A \vect(X)+b}^2+ \lambda \normll{X}$ \\
\hline 
\multicolumn{1}{ |c|| }{ dimension of $X_0$} & $300 \times 300$ & $50 \times 1000$\\
\hline
\multicolumn{1}{ |c|| }{ $K(X_0)$}&${\rm rank}(X_0) = 10$&\# of non-zero columns of $X_0 = 50$\\
\hline
\multicolumn{1}{ |c|| } {$\# {\rm of\, samples}$}&$ m = 20000$&$m = 18000$\\
\hline
\multicolumn{1}{ |c|| } {$b $}&$A \vect(X_0)+z$&$A \vect(X_0)+z$\\
\hline
\multicolumn{1}{ |c|| } {$A_{i,j}$ sampled from}&$\mathcal{N}(0,1/\sqrt{m})$&
$\{-1/\sqrt{m},1/\sqrt{m}\}$ uniformly at rand.\\
\hline
\multicolumn{1}{ |c|| } {$z_i$ sampled from}&$\mathcal{U}(-0.005,0.005)$&$\mathcal{U}(-0.005,0.005)$\\
\hline
\end{tabular}

\end{center}

In the homotopy algorithm, $\lambda_0 =  \dnorm{A^T b}$ and $\lambda_{\rm \rm tgt} =  4 \dnorm{A^T z} $, while in the proximal-gradient algorithm $\lambda = \lambda_{\rm \rm tgt}$. The default values of $\eta$ and $\delta'$ in the homotopy algorithm are $\eta = 0.6$, $\delta' = 0.2$. 

{\bf  Problem 1.}  Figure~\ref{fig:results2} demonstrates the overall linear rate of convergence of proximal-gradient homotopy algorithm (homotopy) applied to this problem and compares it with proximal-gradient algorithm (PG) and its accelerated version (APG). As rank vs. iteration plot demonstrates, the proximal-gradient algorithm speeds up to a linear rate when the rank drops to a certain level, while the homotopy algorithm keeps the rank at a level that ensures a linear rate of convergence.

We examine the performance of homotopy algorithm with three different values of $\eta$ and $\delta'$ in Figure~\ref{fig:results}. For $\eta$ to satisfy the condition of Theorem~\ref{convresult3}, it is necessary that $\eta > 0.5$. However, as Figure~\ref{fig:results} demonstrates, one can choose $\eta \leq 0.5$ and still get an overall linear rate of convergence. For example, when $\eta = 0.2$, at the beginning of the last stage where $\lambda = \lambda_{\rm \rm tgt}$, $X^{(k)}$ is not low-rank and the algorithm has a sublinear rate of convergence, but nevertheless the algorithm converges faster with $\eta = 0.2$ than $\eta = 0.7$. Homotopy algorithm appears to be even less sensitive to $\delta'$. As $\delta'$ gets closer to $1$, the rank of $X^{(k)}$ jumps higher, which can cause a slowdown in convergence specially at the beginning of each stage.

In Figure~\ref{fig:comp2}, we have compared recovery error of the following algorithms: SVP, FPC, APGL, homotopy, proximal-gradient and its accelerated version. In SVP we provide the algorithm with the rank of $X_0$, while in SVP2 we use the same heuristic that is proposed in \cite{jain2010guaranteed} to estimate the rank (other algorithms do not receive the rank of $X_0$).We have implemented the FPC algorithm with the backtracking procedure which improves the performance of the algorithm. Both APGL and APGL2 have been implemented with continuation over $\lambda$ with the latter utilizing an extra truncation heuristic proposed in  \cite{toh2010accelerated}. The method of continuation for APGL is the same as the one proposed in \cite{toh2010accelerated}; we reduce $\lambda$ by a factor of $0.7$ after three iterations or whenever the stopping criterion is met whichever comes first. In FPC and APGL similar to the homotopy algorithms, $\lambda_0 = \dnorm{A^T(b)}$ and $\lambda_{\rm \rm tgt} =  4 \dnorm{A^T(z)} $. We have used the default values of the parameters in all the algorithms. Note that APGL2 has an extra truncation procedure which improves the recovery error. Finally, Figure~\ref{fig:comp1} shows the objective gap for the algorithms for which the quantity is meaningful.

\begin{figure}
       \begin{subfigure}{0.5\textwidth}
                \centering
                \includegraphics[width=.85\textwidth]{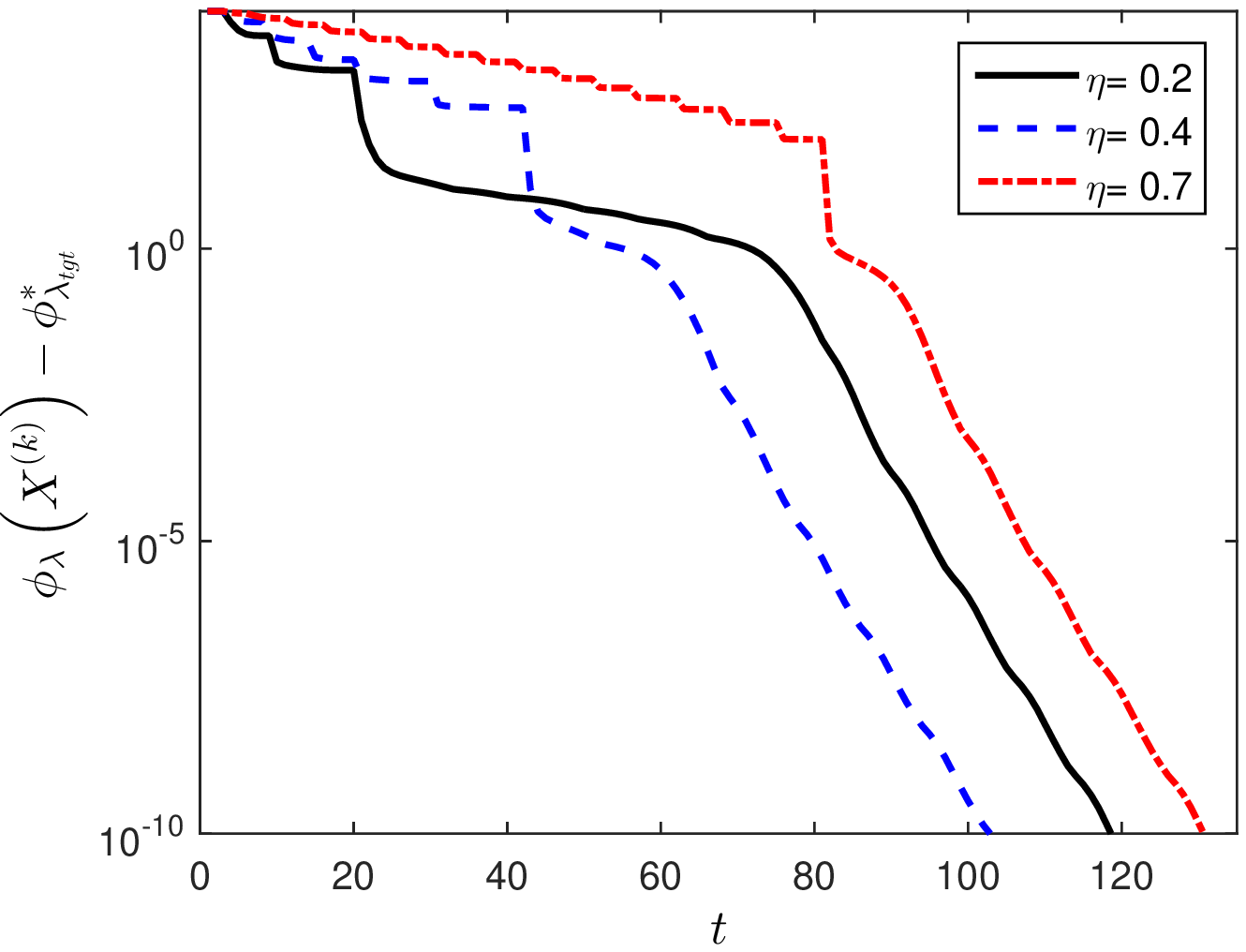}
                \caption{Objective gap vs. iteration}\label{fig:eta}
        \end{subfigure}%
        \begin{subfigure}{0.5\textwidth}
                \centering
                \includegraphics[width=.85\textwidth]{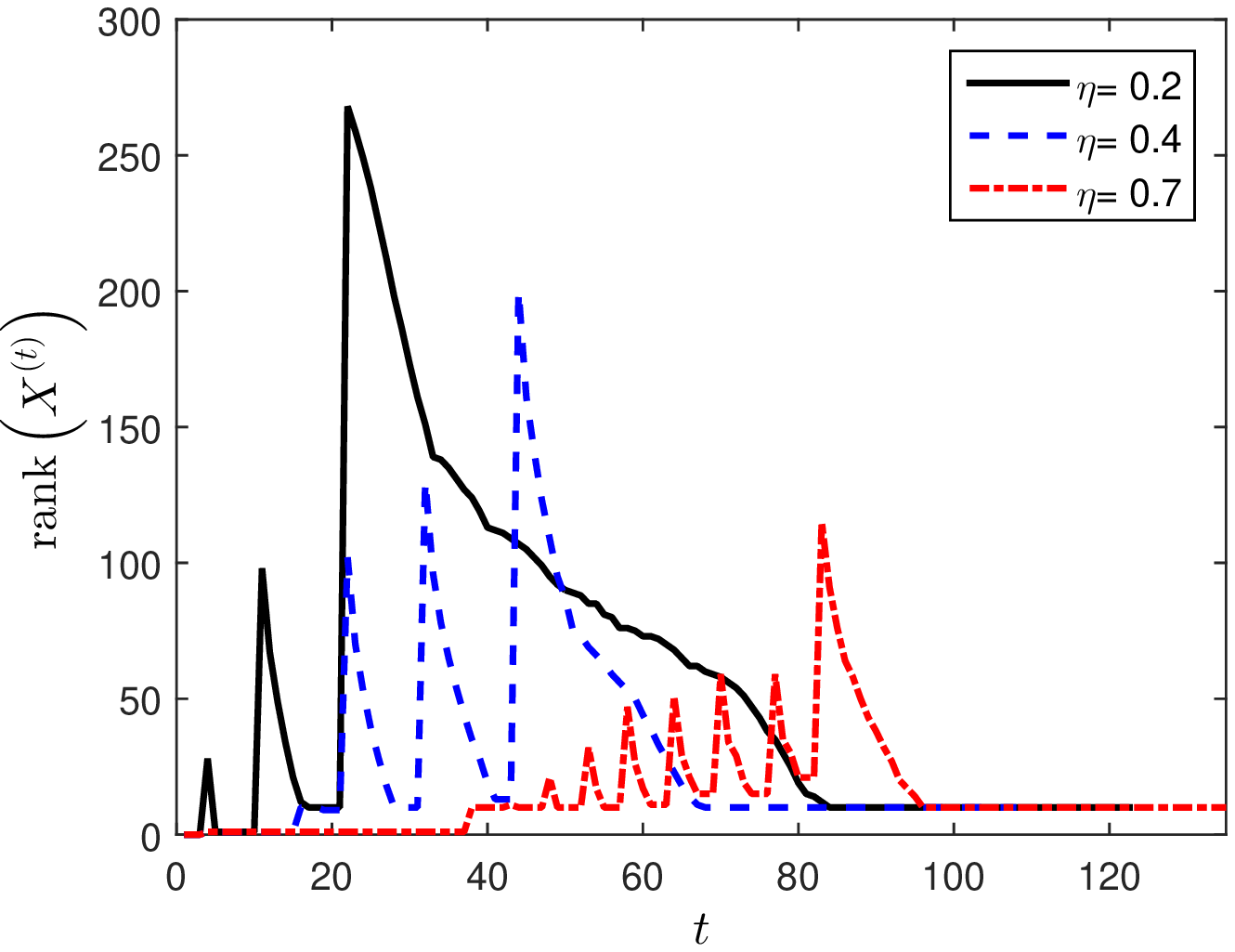}
                 \caption{ Rank vs. iteration}\label{fig:eta3}
                \label{fig:stp}
        \end{subfigure}
        \begin{subfigure}{0.5\textwidth}
                \centering
                \includegraphics[width=.85 \textwidth]{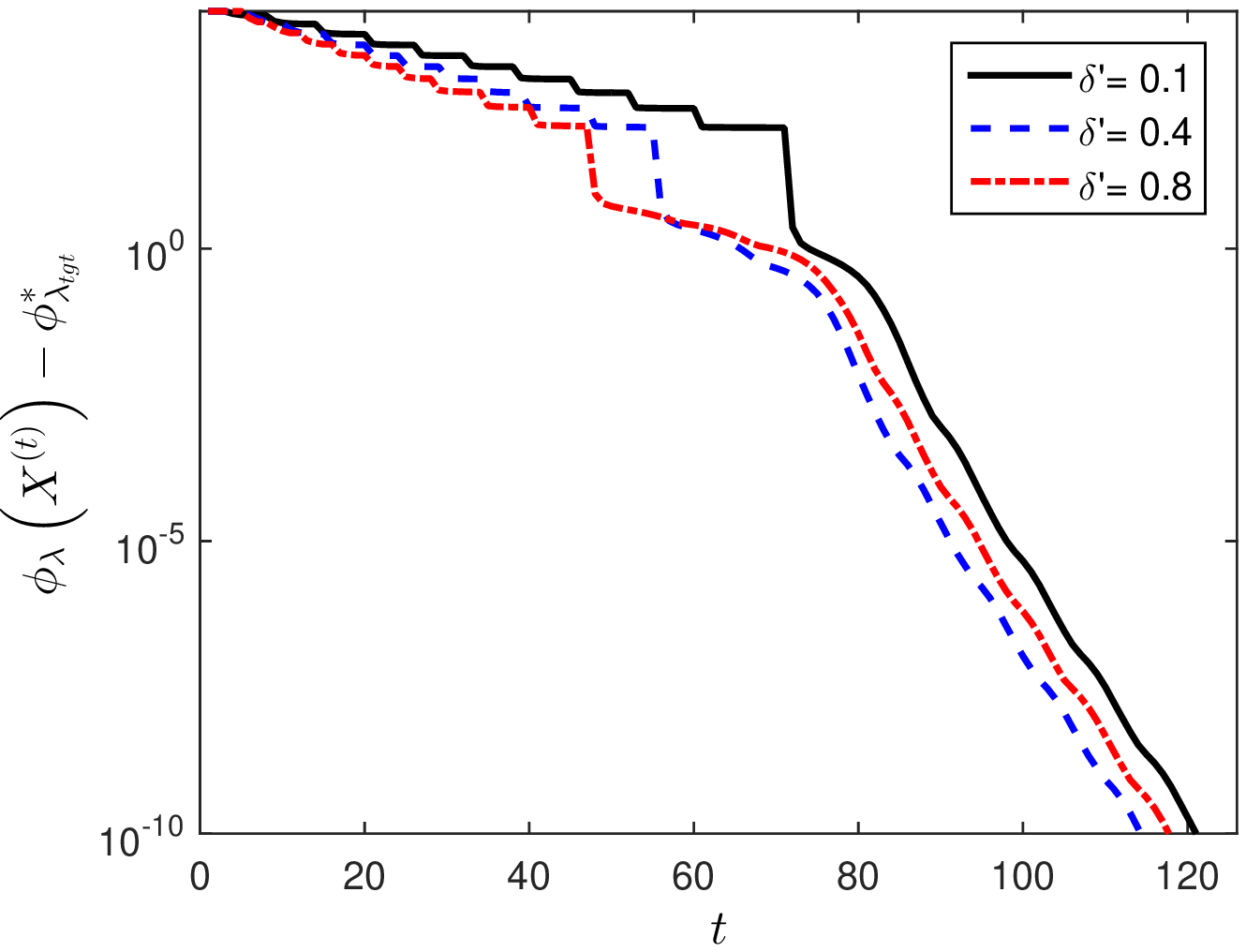}
                \caption{Objective gap vs. iteration}\label{fig:delta}
        \end{subfigure}%
         \begin{subfigure}{0.5\textwidth}
                \centering
                \includegraphics[width=.85\textwidth]{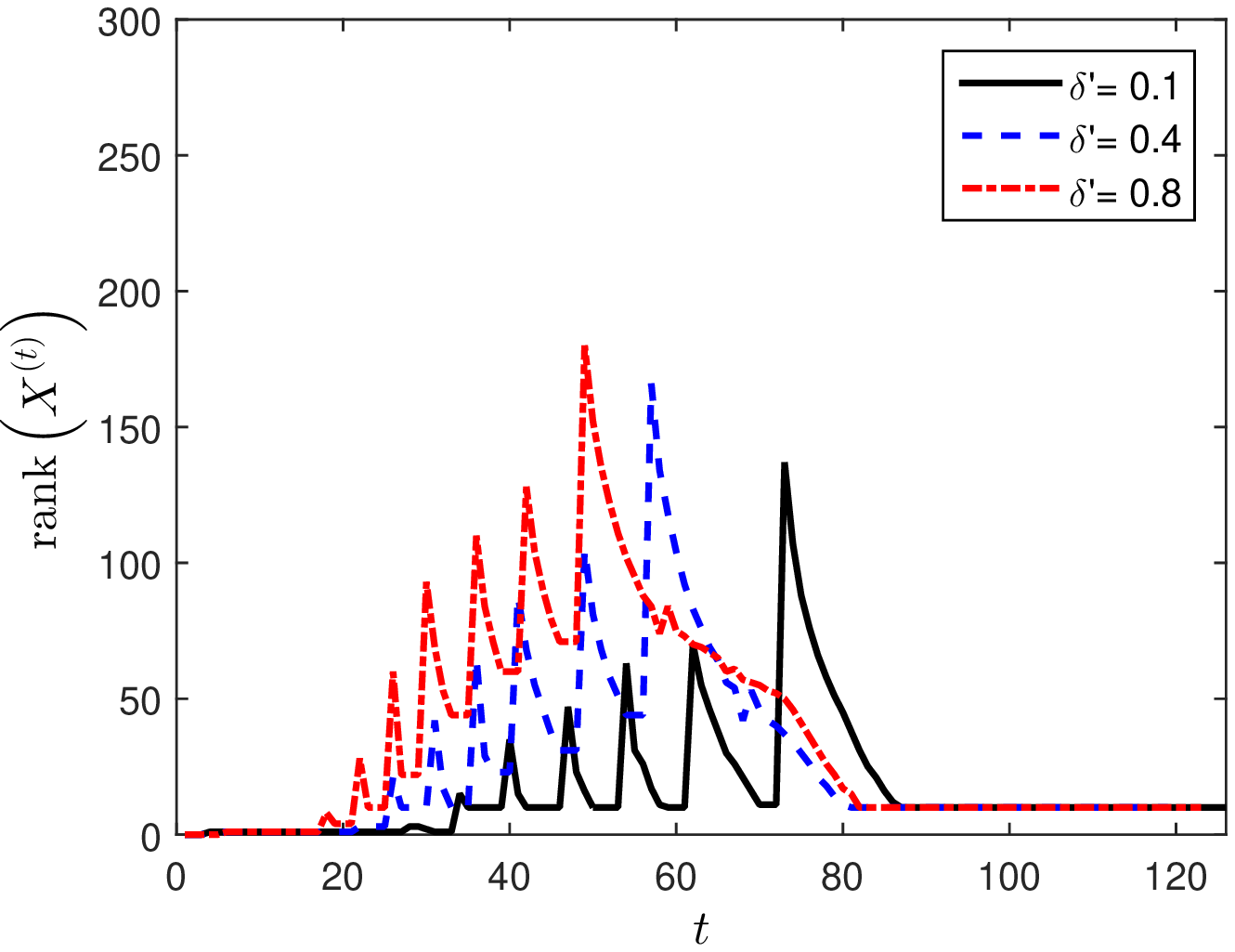}
                 \caption{ Rank vs. iteration}\label{fig:delta3}
                \label{fig:stp}
        \end{subfigure}
        \caption{(a), (b): Performance of homotopy algorithm with $\delta' = 0.2$ and three different values of $\eta$, \\
                       (c), (d): Performance of homotopy algorithm with $\eta = 0.6$ and three different values of  $\delta'$} \label{fig:results}
\end{figure}

\begin{figure}
       \begin{subfigure}{0.5\textwidth}
                \centering

                \includegraphics[width=.85\textwidth]{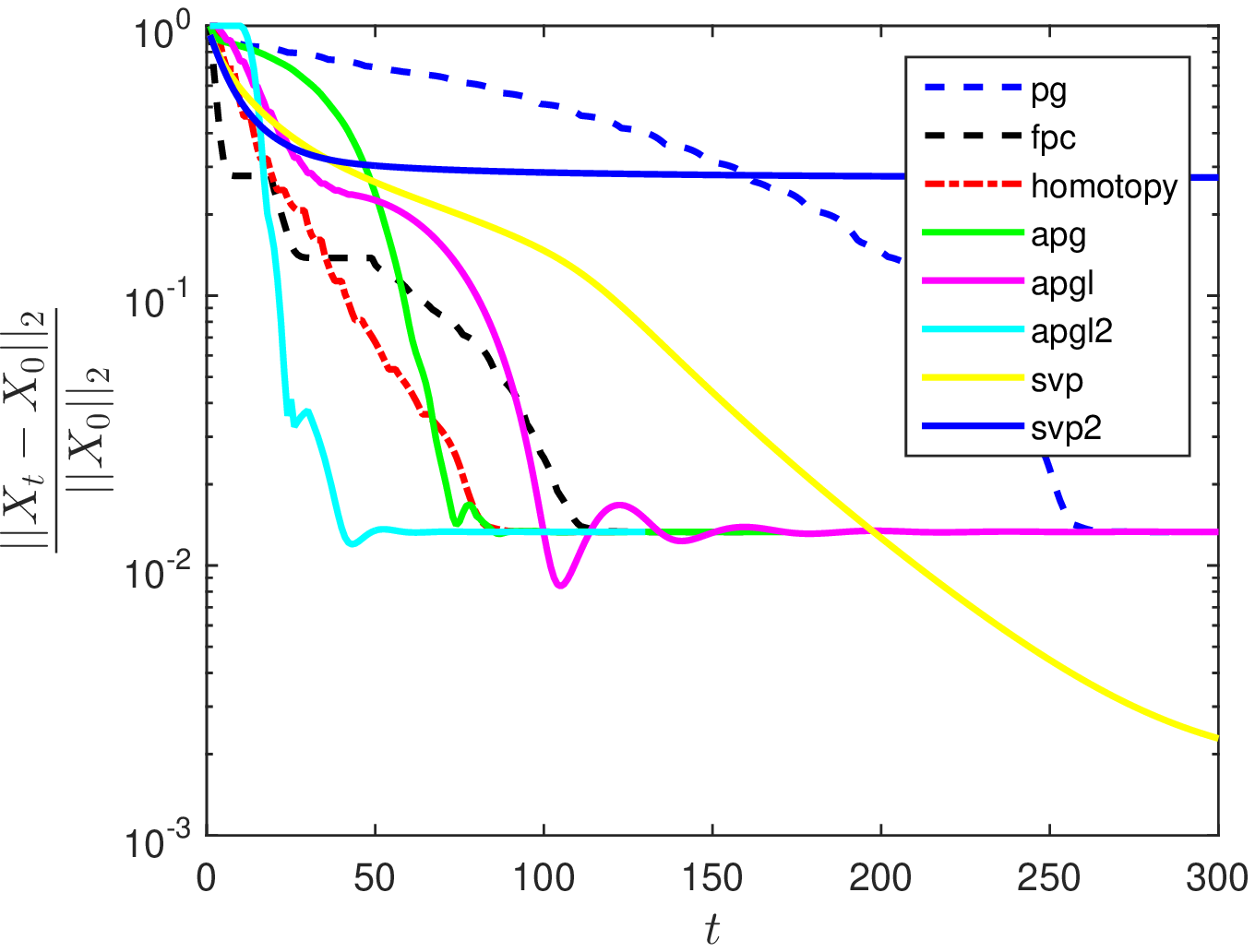}
                 \caption{Recovery error vs. iteration}\label{fig:comp2}

                                            \end{subfigure}%
        \begin{subfigure}{0.5\textwidth}
                  \centering
                
                \includegraphics[width=.85\textwidth]{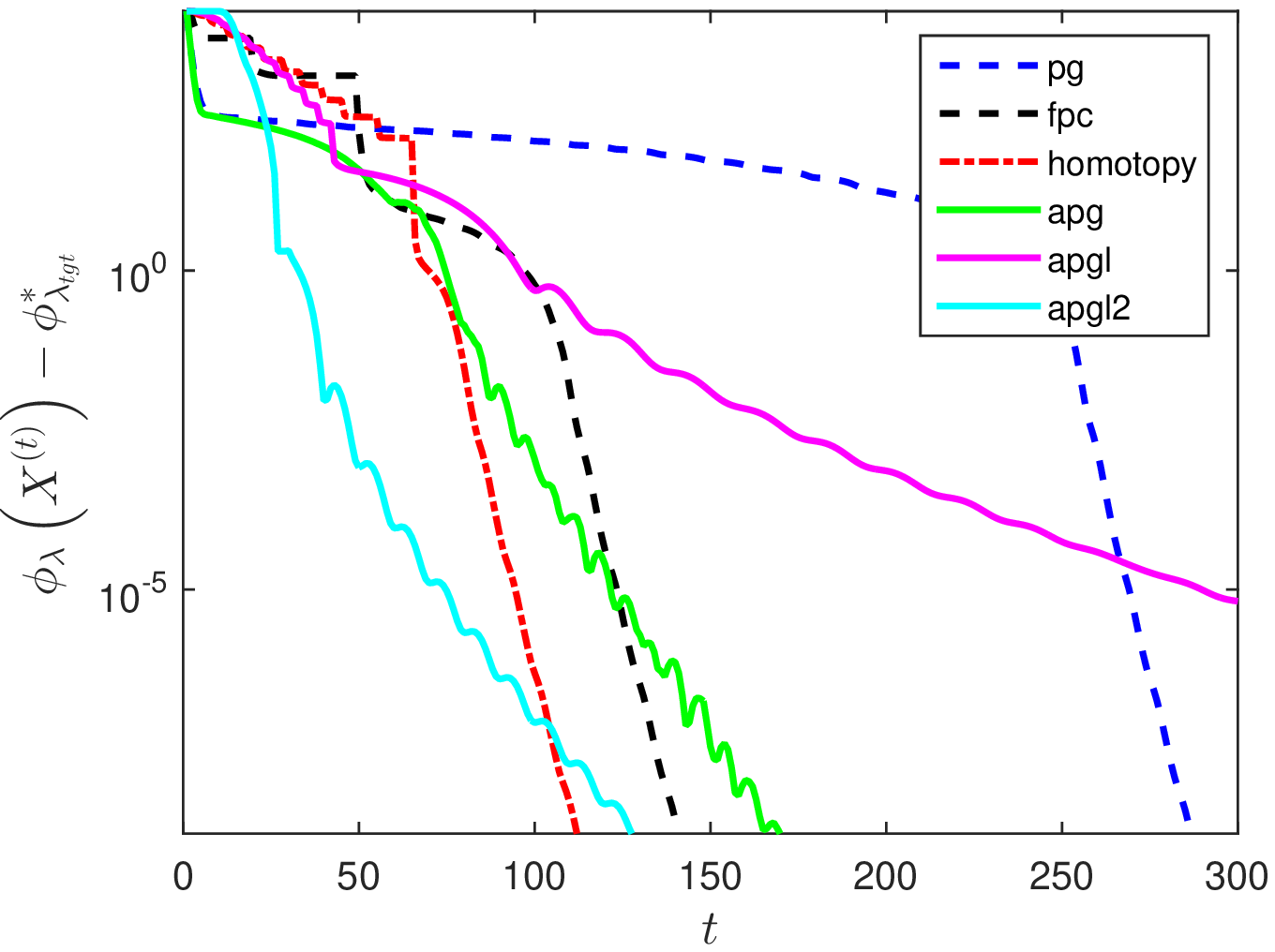}
                \caption{Objective gap vs. iteration}\label{fig:comp1}
        \end{subfigure}
        \caption{Comparison between SVP, FPC, APGL, homotopy, proximal-gradient and its accelerated version}\label{fig:results3}
\end{figure}

{\bf Problem 2.} Figure~\ref{fig:Block} demonstrates the linear convergence of homotopy algorithm for this problem and compares the performance with that of proximal-gradient algorithm and its accelerated version. Similar to problem 1, homotopy algorithm keeps the number of non-zero columns below a certain level. In homotopy algorithm $\delta' = 0.2$ and $\eta = 0.6$.

 
\begin{figure}
\begin{subfigure}{0.5\linewidth}
\centering
\includegraphics[width = .85\textwidth]{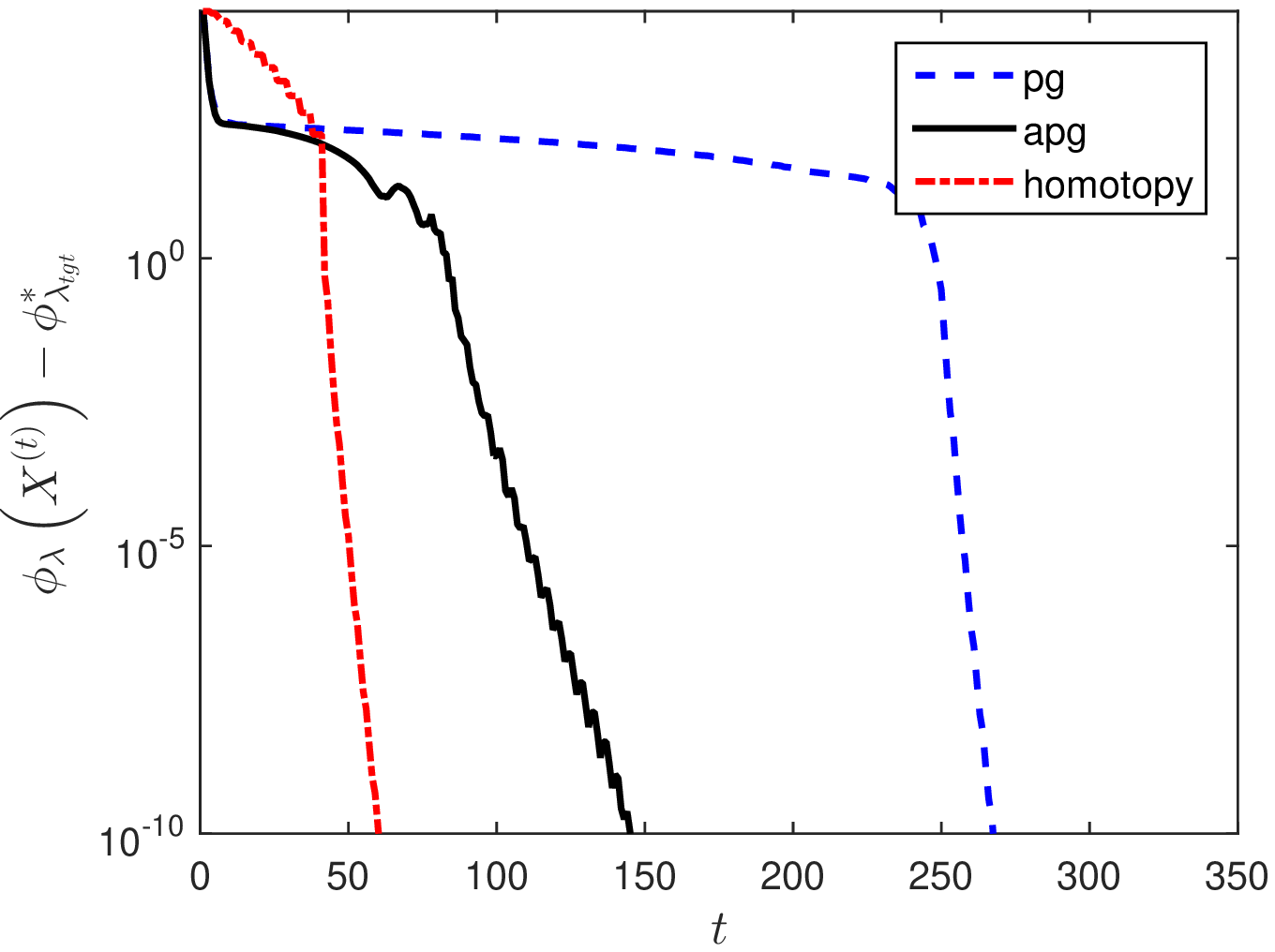}
\caption{Objective gap vs. iteration}\label{fig:Block1}
\end{subfigure}%
\begin{subfigure}{0.5\linewidth}
\centering
\includegraphics[width=.85\textwidth]{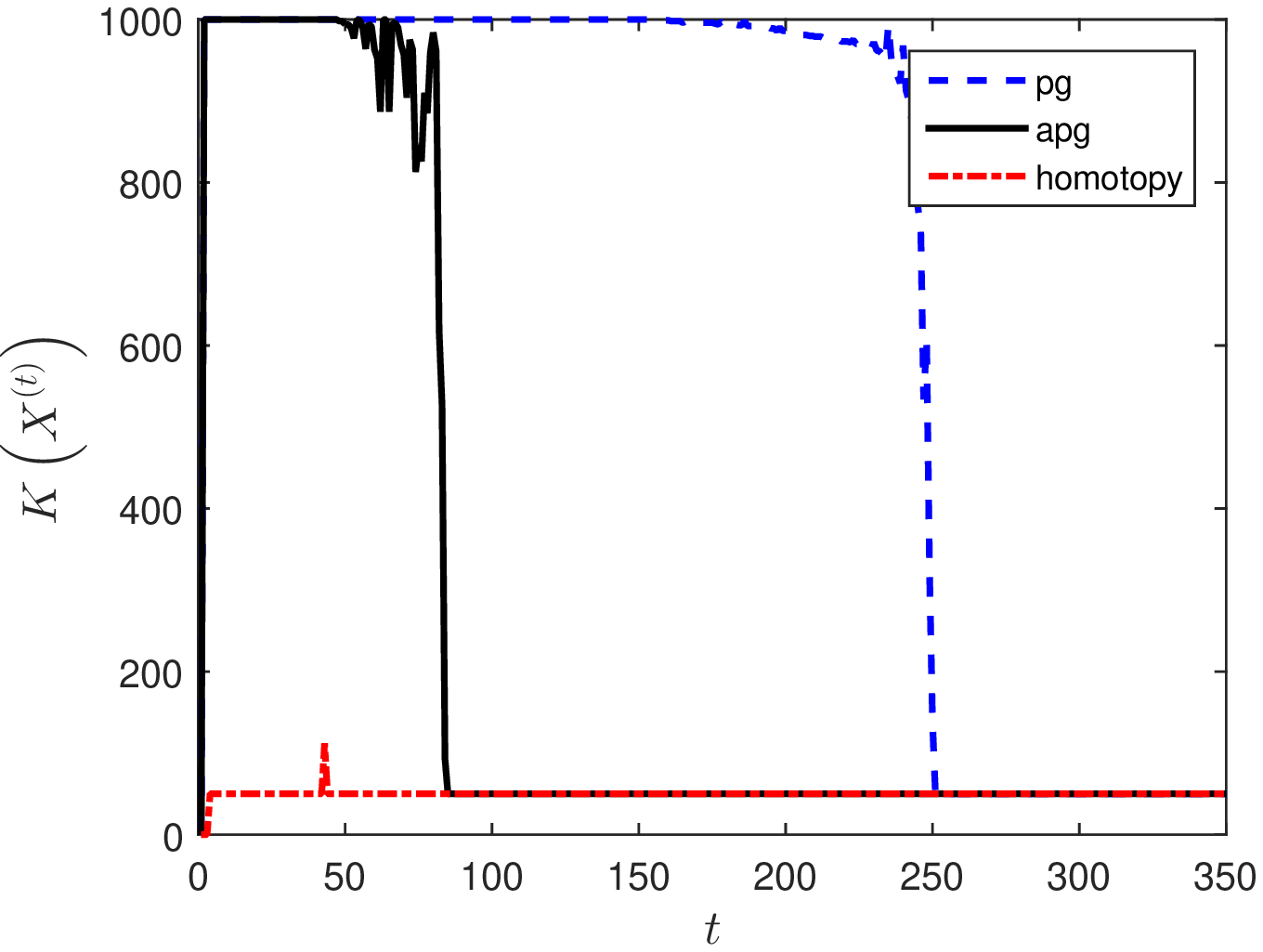}
\caption{Number of non-zero columns vs. iteration}\label{fig:Block2}
\end{subfigure}\\[1ex]
\begin{subfigure}{\linewidth}
\centering
\includegraphics[width=0.425\textwidth]{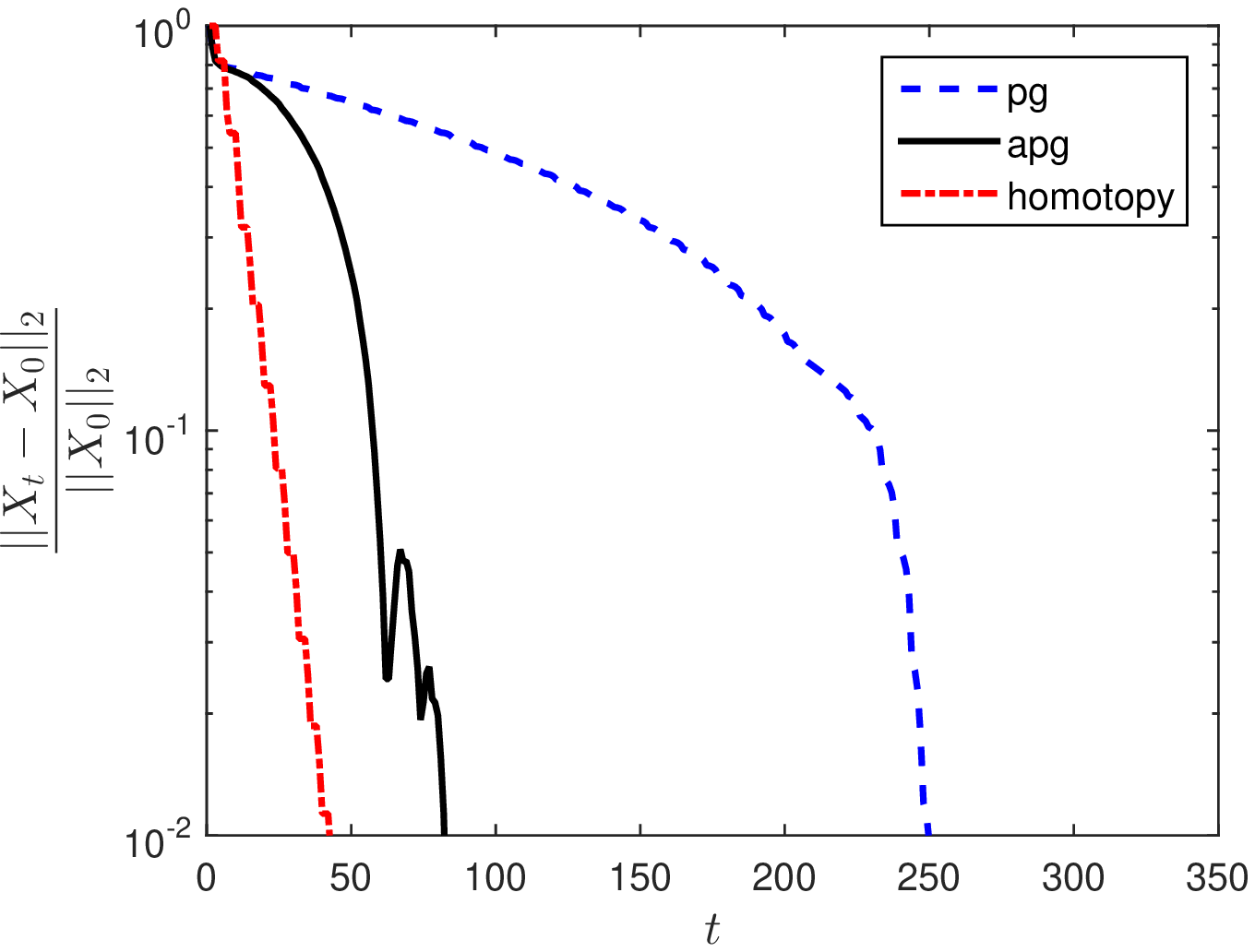}
\caption{ Recovery error vs. iteration}\label{fig:Block3}
\end{subfigure}
\caption{Comparison of homotopy, proximal-gradient and accelerated proximal-gradient algorithms for problem 2}\label{fig:Block}
\end{figure}

%% file: appendix2.tex
\appendix
\section*{Appendix A}

In this section we give a lower bound on the number of measurements $m$ that suffice for the existence of $r>1$ in Assumption~\ref{assumption-2} with high probability when $A$ is sampled from a certain class of distributions. To simplify the notation we assume that $B = I$; therefore, $\inner{x}{y} = x^T y$. Given a random variable $z$ the sub-Gaussian norm of $z$ is defined as:
$$\norm{z}_{\psi_2} = \inf\{\beta> 0 \,| \, \mathbb{E} \psi_2\left(\frac{|z|}{\beta}\right) \leq 1\},$$

\noindent where $\psi_2(x) = e^{x^2}-1$. For an $n$ dimensional random vector $w \sim P$ the sub-Gaussian norm is defined as 
$$\norm{w}_{\psi_2} = \sup_{u \in S^{n-1}} \norm{\inner{w}{u}}_{\psi_2}.$$

$P$ is called isotropic if $E \left[\inner{w}{u}^2\right] = 1$ for all $u \in S^{n-1}$.  Two important examples of sub-Gaussian random variables are Gaussian and bounded random variables. Suppose $A: \mathbb{R}^{n} \mapsto \mathbb{R}^{m}$ is given by:
\begin{align}\label{linear-op}
(Ax)_i = \frac{1}{\sqrt{m}}\inner{A_i}{x} \quad \forall  i \in \{1,2,\ldots,m\},
\end{align}

\noindent where $A_i$, $1\leq i \leq m$ are iid samples from an isotropic sub-Gaussian distribution $P$ on $\mathbb{R}^n$. Two important examples are standard Gaussian vector $A_{i}\sim \mathcal{N}(0,I_n)$ and random vector of independent Rademacher variables \footnote{For general psd $B$, the example are $A_i = B^{-\frac{1}{2}} A_i'  $ with $A_{i}' \sim \mathcal{N}(0,I_n)$ or $A_{i,j}'$ Rademacher for all $j$.}. We want to bound the following probabilities for $\theta \in (0,1)$:

\begin{align}\label{RIPprob}
P(\rho_{-}(A,k)< 1- \theta )\\
P(\rho_{+}(A,k)> 1+ \theta ).
\end{align}

 When $A_{i} \sim \mathcal{N}(0,I_n)$ for all $i$, one can use the generalization of Slepian's lemma by Gordon \cite{gordon1985some} alongside concentration inequalities for Lipschitz function of Gaussian random variable to derive (see, for example, \cite[chapter 15]{ledoux2013probability}):
 \begin{align*}
& P(\sqrt{\rho_{-}(A,k)} < \sqrt{\frac{m}{m+1}} - \theta) \leq e^{-\frac{m {\theta}^2}{8}},\\
 &P(\sqrt{\rho_{+}(A,k)} > 1+ \theta) \leq e^{-\frac{m {\theta}^2}{8}},\\
  \end{align*}
  \noindent whenever,
   \begin{align*}
  {\theta} \geq \frac{2 G(k)}{\sqrt{m}}.
  \end{align*}

Here, $G$ is defined as:
$$G(k) :=  \mathbb{E} \sup_{u \in \sqrt{k} \mathcal{B}_{\norm{\cdot}} \cap S^{n-1}}|\inner{u}{g}|,$$
where $g \sim \mathcal{N}(0,I_n)$. 
For sub-Gaussian case, we use a result by Mendelson et al.\cite[Theorem~2.3]{mendelson2007reconstruction}. Using Talgrand's generic chaining theorem \cite[Theorem 2.1.1]{talagrand2005generic}, the authors have given a result, which similar to the Gaussian case depends on $G(k)$. Their result in our notation states:

\begin{proposition}\label{proposition-1}
Suppose A is given by \eqref{linear-op}. If $P$ is an isotropic distribution and $\norm{A_1}_{\psi_2} \leq \alpha$, then there exist constants $c_1$ and $c_2$ such that
\begin{equation}
\rho_{-}(A/\sqrt{m},k) \geq 1 - \theta ,
\end{equation}
\begin{equation}
\rho_{+}(A/\sqrt{m},k) \leq 1 + \theta ,
\end{equation}

\noindent with probability exceeding $1 - \exp{(-c_2 \theta^2 m / \alpha^4)}$ whenever 
$$\theta \geq  \frac{c_1 \alpha^2 G(k)}{\sqrt{m}}.$$
\end{proposition}

Suppose $\lambda_{\rm tgt} = 4 \dnorm{A^* z}$, which  sets $\gamma = \frac{5+4\delta}{3-4\delta}$. We can state the following proposition based on Proposition \ref{proposition-1} :


{\color{black}{
\begin{proposition}
Let $r>1$, $\tilde{k} =  36 r c k_{0} (1+\gamma)\gamma_{\rm inc}$ and $\bar{k} = c k_0(1+\gamma)^2$. If $m \geq \frac{{c_1} \alpha^4 }{(r-1)^2} (G(2\tilde{k})^2 + r^2 G(\bar{k})^2)$, then $r$ satisfies Assumption~\ref{assumption-2} with probability exceeding $1 - \exp (c_2 (r-1)^2 m / r^2\alpha^2)$.
\end{proposition}

The proof is a simple adaptation of proof of Theorem 1.4  in \cite{mendelson2007reconstruction} which we omit here. To compare this with the number of measurements sufficient for successful recovery within a given accuracy, by combining \eqref{H-first} in the proof Lemma~\ref{conv-lemma-1} and Proposition \ref{proposition-1} we get:}}



\begin{proposition}
Let $r>1$, $\bar{k} = c k_0(1+\gamma)^2$ and $x^* \in \argmin \phi_{\lambda}(x)$. If $m \geq  \frac{{c_1} \alpha^4 r^2}{(r-1)^2} G(\bar{k})^2$, then $\normu{x^* - x_0} \leq  {c_2 r\lambda \sqrt{ck_{0}} }$ with probability exceeding $1 - \exp (c_2 (r-1)^2 m / r^2\alpha^2)$.
\end{proposition}

Note that this bound on $m$ in case of $l_1$, $l_{1,2}$ and nuclear norms orderwise matches the lower bounds given by minimax rates in \cite{raskutti2011minimax}, \cite{lounici2011oracle} and \cite{rohde2011estimation}.

\section*{Appendix B}
\setcounter{section}{2}

\subsection{Proof of Theorem \ref{OrRep}}

{\bf Sufficiency.} First consider the case where $k = 1$ and $x = \gamma_1 a_1$ with $\gamma_1 > 0$. Note that $a_1 \in \partial \norm{x}=\partial \norm{a_1}$ because $\dnorm{a_1}= 1$ for all $a_1 \in \mathcal{G}_{\norm{\cdot}}$ and $\innerB{a_1}{x} = \gamma_1 = \norm{x}$. Define:
\begin{align*}
& C = \{\xi - a_1 | \xi \in \partial \norm{a_1}\}.
\end{align*}
Note that $C$ is a convex set that contains the origin. Moreover, $C$ is orthogonal to $a_1$. We claim that \eqref{decomposable} is satisfied with $T_{a_1}^\bot = \vspan{C}$. To establish the claim, we first prove that $C$ is symmetric and is contained in the dual norm ball. Let $v\in C$ and $\xi = a_1 + v \in \partial{\norm{a_1}}$. By \eqref{norm-sub}, $\innerB{a_1}{\xi} = \dnorm{\xi} =  1$. Therefore, 
\begin{align*}
a_1 \in \argmax_{a \in \mathcal{G}_{\norm{\cdot}}} \inner{a}{\xi}
\end{align*}
and we can apply the hypothesis of the theorem (in particular statement I) to obtain an orthonormal representation for $\xi$: $$\xi = a_1+\sum_{i=1}^{l}{\eta_i b_i}.$$ Now by statement II in the hypothesis  we get:
\begin{align*}
\dnorm{v}=\max_{i}{\eta_i} \leq \dnorm{\xi} \leq 1.
\end{align*}

Let $\xi' = a_1-\sum_{i=1}^{l}{\eta_i b_i}$. By the hypothesis, $\dnorm{\xi'} = \max\{1,\max_{i}{\eta_i}\} = 1$. Also, $\innerB{\xi'}{a_1}=1$ hence $\xi' \in \partial \norm{a_1}$ and $-v \in C$.

Let $v \in \vspan{C}$ with $\dnorm{v}\leq 1$. Since $C$ is a symmetric convex set, there exists $\lambda \in (0,1]$ such that $\lambda v \in C$ (i.e., $C$ is absorbing in $\vspan{C}$). Define $z = a_1+\lambda v$ which is in $\partial{\norm{a_1}}$. Since $\innerB{a_1}{z} = \dnorm{z} = 1$, we can write $z$ as

\begin{equation}\notag
z = a_1 + \sum_{i=1}^{k'}{\nu_i c_i},
\end{equation}

where $\{c_i|i=1,\ldots, k'\} \subset \mathcal{G}_{\norm{\cdot}}$ and $\{\nu_i \geq 0| i=1, \ldots, k'\}$ satisfy the hypothesis of the theorem. In particular, since $v = 1/\lambda\sum_{i=1}^{k'}{\nu_i c_i}$, we have $\max_i{\nu_i/\lambda} \leq 1$. Hence $\dnorm{a_1 + v} = \max\{1,\nu_1/\lambda, \ldots \nu_{t'}/\lambda\} = 1$ and $a_1+v \in \partial{\norm{a_1}}$. Therefore,

\begin{equation}\notag
\partial{\norm{a_1}} = \{a_1+v|v \in \vspan{C}, \dnorm{v}\leq 1\} .
\end{equation}

Now suppose that $x = \sum_{i=1}^{k}{\gamma_i a_i}$ with $k > 1$. Note that ${\sum_{i=1}^{k}{a_i}} \in \partial \norm{x}$ since $\dnorm{\sum_{i=1}^{k}{a_i}} = 1$ and $\innerB{\sum_{i=1}^{k}{a_i}}{x} = \sum_{i=1}^{k}\gamma_i = \norm{x}$. Let $\xi \in \partial{\norm{x}}$ and define $v = \xi - \sum_{i=1}^{k}{a_i}$. We can write:

\begin{align}\notag
\norm{x} = \sum_{i=1}^{k}{\gamma_i} = \innerB{\xi}{x} = \sum_{i=1}^{k}{\gamma_i \innerB{\xi}{a_i}}\\ \label{subgradi}
\Rightarrow \forall i\in \{1,2,\ldots,k\}: \, \innerB{\xi}{a_i} = 1 \Rightarrow \forall i \in \{1,2,\ldots,k\}:\, \xi \in \partial{\norm{a_i}}.
\end{align}
 
Also, since  $\sum_{i=1}^{k}{a_i} \in \partial \norm{a_i}$, \eqref{subgradi} results in:

\begin{equation}\label{vTperp}
\forall i\in \{1,2,\ldots,k\}:  \quad v \in {T_{a_{i}}^{\bot}}.
\end{equation}

Since $\xi = \sum_{i=1}^{k}{a_i} + v \in \partial\norm{a_1}$, we have $\dnorm{\sum_{i=2}^{k}{a_i} + v} = 1$ hence $\sum_{i=2}^{k}{a_i} + v \in \partial\norm{a_2}$. By induction, we conclude that $a_k + v \in \partial \norm{a_k}$. This implies $\dnorm{v}\leq 1$.

Let $v' \in \cap_{i \in \{1,2,\ldots,k\}}{T_{a_{i}}^{\bot}}$ with $\dnorm{v'}\leq 1$ and define $\xi' = \sum_{i=1}^{k}{a_i}+ v'$.
We will prove that $\dnorm{\xi'} \leq 1$ and hence $\xi' \in \partial{\norm{x}}$ . To prove this we use induction.
Define 
$$z_l = \sum_{i=k-l+1}^{k}{a_i}+v' \qquad \forall l \in \{1,2,\ldots,k\}.$$
Note that $\dnorm{z_1} \leq 1$ since $z_1 = a_k+v' \in \partial \norm{a_k}$. Suppose $\dnorm{z_{l'}} \leq 1$ for some $l' < k $. We prove that $\dnorm{z_{l'+1}} \leq 1$. We have $\sum_{i=k-l'+1}^{k}{a_i} \in T_{a_{k-l'}}^{\bot}$ because $\sum_{i=k-l'}^{k}{a_i}  = a_{k-l'} + \sum_{i=k-l'+1}^{k}{a_i}  \in \partial \norm{a_{k-l'}}$. Combining this with the fact that $v' \in T_{a_{k-l'}}^{\bot}$, we get $z_{l'} \in T_{a_{k-l'}}^{\bot}$. Therefore, $z_{l'+1}= a_{k-l'}+z_{l'} \in \partial \norm{a_{k-l'}}$ hence $\dnorm{z_{l'+1}}\leq 1$. Thus $\dnorm{\xi'} = \dnorm{z_k}\leq 1$. We conclude that:
\begin{align}
\partial& \norm{x} = \{\sum_{i=1}^{k}{a_i}+v| v \in \bigcap_{i=1}^{k}{T_{a_i}^{\bot}}, \dnorm{v}\leq 1\}.
\end{align}

\noindent{\bf Necessity.}
For any $a\in \mathcal{G}_{\norm{\cdot}}$, we have:
\begin{align}\notag
&\innerB{a}{a} = 1, \\ \notag
\forall b \in \mathcal{G}_{\norm{\cdot}}: \quad & \innerB{b}{a} \leq \normu{b}\normu{a}= 1.
\end{align}

That implies $\dnorm{a} = 1$ and $a \in \partial \norm{a}$. Since $a\in T_a$, we conclude that:

\begin{equation}\label{subgrad-form}
\partial \norm{a} = \{a+v|v \in {T_{a}^{\bot}}, \dnorm{v}\leq 1\}.
\end{equation}

Take $\gamma_1 = \innerB{a_1}{x}=\dnorm{x}$ and let $\Delta_1 = x - {\gamma_1} a_1$. If $\Delta_1 = 0$, then take $k  =1$ and $x = \gamma_1 a_1$. Suppose $\Delta_1 \neq 0$. Since $\dnorm{\frac{1}{\gamma_1}x} = 1$ and $\innerB{a_1}{{\frac{1}{\gamma_1}x}} = \norm{a_1} = 1$, we can conclude that $\frac{1}{\gamma_1}x  \in \partial \norm{a_1}$. Furthermore, we have

\begin{align}\label{PTreason}
\ptpx{a_1}{x} = x - \gamma_1 \ptx{a_1}{\frac{1}{\gamma_1}x} = x - {\gamma_1} a_1= \Delta_1
\end{align}
$$\Rightarrow \Delta_1 \in T_{a_1}^{\bot}.$$

Now we introduce a lemma that will be used in the rest of the proof.
\begin{lemma}\label{interm}
Suppose $a \in \mathcal{G}_{\norm{\cdot}}$ and $y \in T^{\perp}_{a} - \{0\}$. If $z \in \mathcal{B}_{\norm{\cdot}}$ is such that $\dnorm{y} = \innerB{y}{z}$, then $z \in T^{\perp}_{a}$.
\end{lemma}
\begin{proof}
Without loss of generality assume that $\dnorm{y} = 1$. It suffices to show that if $b \in \mathcal{G}_{\norm{\cdot}}$ and $ \innerB{y}{b} = 1$, then $b \in T^{\perp}_{a}$. Consider such $b \in \mathcal{G}_{\norm{\cdot}}$. By \eqref{subgrad-form}, $\dnorm{a+y} = 1$. That results in:
\begin{equation}\notag
1 \geq \innerB{a+y}{b} =  \innerB{a}{b} +1  \Rightarrow   0 \geq \innerB{a}{b}.
\end{equation}

By considering $-y$ and $-b$ we get that $\innerB{a}{b} = 0$.  Since $\innerB{a+y}{b} = \norm{b} = 1$, we can conclude that $a+y \in \partial \norm{b}$. Since $ \innerB{y}{b} = 1$ and $\dnorm{y} = 1$, $y \in \partial \norm{b}$. Combining these two conclusions, we get:

\begin{align*}
&y \in \partial \norm{b}, a+y \in \partial \norm{b} \Rightarrow a \in  T^{\perp}_{b} \xRightarrow{\eqref{subgrad-form}} \dnorm{a+b} \leq 1
\Rightarrow a+b \in \partial \norm{a} \Rightarrow b \in T^{\perp}_{a}
\end{align*}
\qed
\end{proof}


Suppose that there exist $l \in \{1,2,\ldots,k\}$, an orthogonal set $\{a_i \in \mathcal{G}_{\norm{\cdot}}|\,  i=1,2,\ldots, l\}$, and a set of coefficients $\{\gamma_i \geq 0 |\, i = 1,2,\ldots,l\}$ such that $x = \sum_{i=1}^{l}{\gamma_i a_i} + \Delta_l$, $\Delta_l \in \cap_{i=1}^{l}{T_{a_i}^{\bot}} $, and:

\begin{align}\label{ind2}
\partial& \norm{\sum_{i=1}^{l}{a_i}} = \{\sum_{i=1}^{l}{a_i}+v|v \in \bigcap_{i=1}^{l}{T_{a_i}^{\bot}}, \dnorm{v}\leq 1\}.
\end{align}
  By Lemma \ref{interm}, there exists $a_{l+1} \in \mathcal{G}_{\norm{\cdot}}$ such that $a_{l+1} \in \cap_{i=1}^{l}{T_{a_i}^{\bot}}$ and $\innerB{a_{l+1}}{\Delta_l} = \dnorm{\Delta_l}$. Take $\gamma_{l+1} = \innerB{a_{l+1}}{\Delta_l} =  \dnorm{\Delta_l}$ and let $\Delta_{l+1} = \Delta_l - \gamma_{l+1}  a_{l+1}$. We have $\Delta_{l+1} \in \bigcap_{i=1}^{l}{T_{a_i}^{\bot}}$ because $\{\Delta_l, a_{l+1}\} \subset \bigcap_{i=1}^{l}{T_{a_i}^{\bot}}$. Since $\dnorm{\frac{1}{\gamma_{l+1}}\Delta_{l}} = 1$ and $\innerB{a_{l+1}}{{\frac{1}{\gamma_{l+1}}\Delta_{l}}} = \norm{a_{l+1}} = 1$, we can conclude that $\frac{1}{\gamma_{l+1}}\Delta_l  \in \partial \norm{a_{l+1}}$. Using the same reasoning as in \eqref{PTreason}, we have $\Delta_{l+1} \in T_{a_{l+1}}^{\bot}$ hence $ \Delta_{l+1} \in \bigcap_{i=1}^{l+1}{T_{a_i}^{\bot}}$.

  
By decomposability assumption there exists $e\in \mathbb{R}^n$ and a subspace $T$ such that:

\begin{equation}
\partial \norm{\sum_{i=1}^{l+1}{a_i}} = \{e+v|v \in {T^{\bot}}, \dnorm{v}\leq 1\}.
\end{equation}

We claim that 
\begin{align}\label{first-claim}
e &= \sum_{i=1}^{l+1}{a_i}\\ \label{second-claim}
{T^{\bot}} &= \bigcap_{i=1}^{l+1}{T_{a_i}^{\bot}}.
\end{align}
To prove the first claim, it is enough to show that $\sum_{i=1}^{l+1}{a_i} \in \partial \norm{\sum_{i=1}^{l+1}{a_i}}$. Note that $ \dnorm{\sum_{i=1}^{l+1}{a_i} }\leq 1$ since $\sum_{i=1}^{l+1}{a_i}= \sum_{i=1}^{l}{a_i}+a_{l+1} \in \partial \norm{\sum_{i=1}^{l}{a_i}}$ which is given by \eqref{ind2}. Now we can write:

$$l+1 = \innerB{\sum_{i=1}^{l+1}{a_i}}{\sum_{i=1}^{l+1}{a_i}}\leq \norm{\sum_{i=1}^{l+1}{a_i}} \dnorm{\sum_{i=1}^{l+1}{a_i}} \leq \norm{\sum_{i=1}^{l+1}{a_i}}.$$
On the other hand, by triangle inequality,
$$\norm{\sum_{i=1}^{l+1}{a_i}} \leq \norm{\sum_{i=1}^{l}{a_i}}+\norm{a_{l+1}}=l+1,$$
thus
$$\norm{\sum_{i=1}^{l+1}{a_i}} = \innerB{\sum_{i=1}^{l+1}{a_i}}{\sum_{i=1}^{l+1}{a_i}}.$$

Therefore,  $\sum_{i=1}^{l+1}{a_i} \in \partial{\norm{\sum_{i=1}^{l+1}{a_i}}}$. Since $\sum_{i=1}^{l+1}{a_i}\in T_{\sum_{i=1}^{l+1}{a_i}} = T$, we conclude that:
$$\partial \norm{\sum_{i=1}^{l+1}{a_i}} = \{\sum_{i=1}^{l+1}{a_i}+v|v \in {T^{\bot}}, \dnorm{v}\leq 1\}.$$

To prove \eqref{second-claim}, we first show that $\bigcap_{i=1}^{l+1}{T_{a_i}^{\bot}} \in T^{\bot}$. Let $\xi =  e+v$ with $v \in \bigcap_{i=1}^{l+1}{T_{a_i}^{\bot}}$. Note that $\dnorm{a_{l+1}+v}\leq 1$ since $a_{l+1}+v \in \partial{\norm{a_{l+1}}}$. Furthermore, $a_{l+1}+v \in \bigcap_{i=1}^{l}{T_{a_i}^{\bot}}$, which in turn implies ${\sum_{i=1}^{l+1}{a_i}+v} \in \partial{\norm{\sum_{i=1}^{l}{a_i}}}$ hence $\dnorm{\sum_{i=1}^{l+1}{a_i}+v} \leq 1$. Additionally, we have:
$$\innerB{\sum_{i=1}^{l+1}{a_i}+v}{\sum_{i=1}^{l+1}{a_i}} = \norm{\sum_{i=1}^{l+1}{a_i}} = l+1.$$
Hence $\xi \in \partial \norm{\sum_{i=1}^{l+1}{a_i}}$ and $v \in T^{\bot}$. 

Now, let $\xi' = \sum_{i=1}^{l+1}{a_i}+v' \in \norm{\sum_{i=1}^{l+1}{a_i}}$. Note that:

$$\innerB{\xi'}{\sum_{i=1}^{l+1}{a_i}} = \innerB{\xi'}{\sum_{i=1}^{l}{a_i}}+\innerB{\xi'}{{a_{l+1}}} = l+1 \Rightarrow \innerB{\xi'}{\sum_{i=1}^{l}{a_i}} = l, \innerB{\xi'}{{a_{l+1}}}=1 \Rightarrow \xi' \in \partial \norm{\sum_{i=1}^{l}{a_i}} \cap \partial \norm{{a_{l+1}}}$$
$$\xRightarrow{} v' \in\bigcap_{i=1}^{l}{T_{a_i}^{\bot}} , \sum_{i=1}^{l}{a_i} + v' \in \bigcap_{i=1}^{l}{T_{a_i}^{\bot}};$$

\noindent moreover, $\sum_{i=1}^{l}{a_i} \in T_{a_{l+1}}^{\bot}$ since $\sum_{i=1}^{l+1}{a_i} \in \partial \norm{a_{l+1}}$. This implies $v \in \bigcap_{i=1}^{l+1}{T_{a_i}^{\bot}}$ which completes the proof of \eqref{second-claim}.

Because $a_i \notin T_{a_i}^{\bot}$ for all $i \in \{1,2,\ldots,l+1\}$, $dim(\cap_{i = 1}^{l+1}T_{a_i}^{\bot} )\leq n-l-1 $. Hence there exists $k \leq n$, an orthogonal set $\{a_i \in \mathcal{G}_{\norm{\cdot}},\,  i = 1, 2, \ldots, k\}$, and a set of coefficients $\{\gamma_i \geq 0 \text,\, i\in\{1,2,\ldots,k\}\}$ such that $x = \sum_{i=1}^{k}{\gamma_i a_i}$ and:

\begin{align}\label{conclusion}
\partial& \norm{\sum_{i=1}^{k}{a_i}} = \{\sum_{i=1}^{k}{a_i}+v\;|\; v \in \bigcap_{i=1}^{k}{T_{a_i}^{\bot}}, \dnorm{v}\leq 1\}.
\end{align}

That proves $\norm{x} = \innerB{\sum_{i=1}^{k}{a_i}}{x} = \sum_{i=1}^{k}{\gamma_i}.$ 

%
%
%
%
%
%

 To prove statement II, we first prove that $a_i \in T_{a_j}^{\bot}$ for all $i,j \in \{1,2,\ldots,k\}$. By \eqref{conclusion}, $\dnorm{\sum_{i=1}^{k}{a_i}}\leq 1$. We can write:

\begin{equation}\notag
\innerB{\sum_{i=1}^{k}{a_i}}{a_j} = 1 \Rightarrow \sum_{i=1}^{k}{a_i} \in \partial \norm{a_j} \Rightarrow \sum_{i=1}^{k}{a_i} - a_j \in T_{a_j}^{\bot}, \; \dnorm{\sum_{i=1}^{k}{a_i} - a_j}\leq 1,
\end{equation}
Now the claim follows from Lemma \ref{interm}.

Let $l=| \{\eta_i|\eta_i \neq 0\}|$. If $l = 0$, the statement is trivially true. Suppose the statement is true when $l=l'-1$ for some $l' \in \{1,\ldots,n\}$ and consider the case where $l=l'$. Suppose that $|\eta_j| = \max_{i}{|\eta_i|}$. By proper normalization we can assume that $\eta_j = 1$. Let $y = \sum_{i\neq j}{\eta_ia_{i}}$. We can deduce the following properties for $y$:

\begin{align*}
& \forall i\neq j: \, a_i \in T_{a_j}^{\bot} \Rightarrow y \in  T_{a_j}^{\bot},\\
& \dnorm{y} = \max_{i\neq j}{|\eta_i|} \leq 1.
\end{align*}

By the decomposability assumption $\sum_{i=1}^{k}{\eta_ia_{i}}= a_j + y \in \partial \norm{a_j}$ hence $\dnorm{\sum_{i =1}^{k}{\eta_ia_{i}}}\leq 1$. Hence $\dnorm{\sum_{i=1}^{k}{\eta_ia_{i}}} = 1$.\qed 

\begin{remark}
 Let $x = \sum_{i=1}^{K(x)}\gamma_i a_i$. Since ${T^{\bot}_{x}} = \bigcap_{i=1}^{K(x)}{T_{a_i}^{\bot}}$, a more general version of lemma \ref{interm} holds:

\begin{lemma}\label{interm2}
Suppose $x \in \mathbb{R}^n$ and $y \in T^{\perp}_{x} - \{0\}$. If $z \in \mathcal{B}_{\norm{\cdot}}$ is such that $\dnorm{y} = \innerB{y}{z}$, then $z \in T^{\perp}_{x}$.
\end{lemma}

We state and prove a dual version of  Lemma \ref{interm2}, which will be used in the proof of  Lemma \ref{conv-lemma-1} and Lemma \ref{conv-lemma-2}.

\begin{lemma}\label{dual}
Let $x \in \mathbb{R}^n$. If  $y \in T^{\perp}_{x}$, then there exists $z \in T_{x}^{\bot} \cap \mathcal{B}_{\dnorm{\cdot}}$ such that $\norm{y} = \innerB{y}{z}$.

\end{lemma}

\begin{proof}
If $y = 0$, then the lemma is trivially true. If $y \neq 0$, then:
\begin{align*}
&\frac{y}{\norm{y}} \in  T_{x}^{\bot} \cap \{x\;|\; \norm{x} = 1\} \Rightarrow \exists z \in T_{x}^{\bot} \text{  such that  }  \frac{y}{\norm{y}} \in \argmax_{a \in T_{x}^{\bot} \cap \mathcal{B}_{\norm{\cdot}}}{\innerB{a}{z}}.
\end{align*}

Therefore, by Lemma \ref{interm2}, we get
\begin{align*}
\dnorm{z} = \max_{a \in T_{x}^{\bot} \cap \mathcal{G}_{\norm{\cdot}}}{\innerB{a}{z}} \leq  \innerB{\frac{y}{\norm{y}}}{z} \leq \dnorm{z}  \Rightarrow \innerB{\frac{y}{\norm{y}}}{z} = \dnorm{z} \Rightarrow \innerB{y}{\frac{z}{\dnorm{z}}}=\norm{y}
\end{align*}
\qed
\end{proof}

\end{remark}
\subsection{Proof of Theorem \ref{thm-con-2}}

First, we introduce a lemma.

\begin{lemma}\label{arbit-sum-primal}
Let $\{a_1,\ldots, a_k\}$ be an orthogonal subset of $\mathcal{G}_{\norm{\cdot}}$ that satisfies II in Theorem \ref{OrRep}. Let $y = \sum_{i=1}^{k}\beta_i a_i$, with $\beta_i \in \mathbb{R}$ for all $i$, then $$K(y) = |\{i\;|\; \beta_i \neq  0\}|.$$
\end{lemma}

\begin{proof}
Let $k' = |\{i\;|\; \beta_i \neq  0\}|$. Without loss of generality assume that $\beta_i \neq 0$ for $i \leq k'$ and $\beta_i = 0$ for $i > k'$. Let $\eta_i = {\rm sgn}(\beta_i)$ and $a'_i = {\rm sgn}(\beta_i) a_i$ for all $i \leq k'$. Since $a_1, \ldots, a_k$ satisfy condition II in the orthogonal representation theorem, so do $a'_1,\ldots, a'_{k'}$.

Now we show that $y$ and $a'_1,\ldots, a'_{k'}$ satisfy condition I. By \eqref{arbit-sum}, $\dnorm{\sum_{i=1}^{k'} a'_i}  \leq 1$. Therefore,
\begin{align*}
&\norm{y} \geq \inner{\sum_{i=1}^{k'} a'_i}{y} = \sum_{i=1}^{k'} |\beta_i|, \quad
\norm{y} = \norm{ \sum_{i=1}^{k'} \beta_i a'_i} \leq \sum_{i=1}^{k'} |\beta_i| \quad \Rightarrow \quad \norm{y} = \sum_{i=1}^{k'} |\beta_i|.
\end{align*}

Therefore, by the orthogonal representation theorem, $e_y = \sum_{i=1}^{k'} a'_i$. Thus $K(y) = \normu{e_y}^2 = k'$.
\qed
\end{proof}

For any $x \in \mathbb{R}^n - \{0\}$ define

$$l(x) = \min \{l \; | \; x = \sum_{i=1}^{l} \alpha_i b_i, \; b_1, \ldots, b_l \subseteq \mathcal{G}_{\norm{\cdot}}, \;\alpha_i \in \mathbb{R} \}.$$

Define $l(0) = 0$. Now the proof is a simple consequence of the following lemma:

\begin{lemma}  For all $x \in \mathbb{R}^n$, $l(x) = K(x)$.
\end{lemma}

\begin{proof}

$K(x) \geq l(x)$ by the definition of $l(x)$. We prove that $K(x) = l(x)$ by induction on $K(x)$. When $K(x) \in\{0,1\}$, the statement is trivially true. Suppose the statement is true when $K(x) \in \{0,1,2,\ldots,k-1\}$. Consider the case where $K(x) = k$. By way of contradiction, suppose $l(x) < K(x)$. Let 
\begin{align}\label{k-rep}
x = \sum_{i=1}^{k} \gamma_i a_i,
\end{align} 
where $\gamma_1, \ldots, \gamma_k$ and $a_1,\ldots, a_k$ are given by the orthogonal representation theorem. If $l(x) = 1$, then:
$$\sum_{i=1}^{k} \gamma_i a_i = \alpha_1 b_1,$$

\noindent for some $\alpha_1 \neq 0$ and $b_1 \in \mathcal{G}_{\norm{\cdot}}$. Since $\lvert \alpha_1\lvert = \norm{ \alpha_1 b_1} = \norm{x} = \sum_{i=1}^{k} \gamma_i $, either $b_1$ or $-b_1$ can be written as convex combination of $a_1, \ldots, a_k$ which  contradicts the fact that $b_1\in  \mathcal{G}_{\norm{\cdot}}$.
%
%

If $l(x) = l > 1$, we can write $x$ as:

\begin{align}\label{l-rep}
x =  \sum_{i=1}^{l} \alpha_i b_i,
\end{align}
\noindent with $\{b_1, \ldots, b_l\} \subseteq \mathcal{G}_{\norm{\cdot}}$. By turning $b_i$ to $- b_i$ without loss of generality we assume that $\alpha_i > 0$ for all $i$. Let $u = 2 \alpha_1 b_1 $ and $v = 2 \sum_{i=2}^{l} \alpha_i b_i$ and note that $x = (u+v)/2$. Let $C = {\rm Cone}{\{a_1, a_2,\ldots, a_k\}}$. Let ${\rm int} C$ and ${\rm bd} C$ denote the interior and the boundary of $C$, respectively. Note that $u \notin {\rm int C}$ because by Lemma \ref{arbit-sum-primal}, if $u \in {\rm int C}$, then $K(u) = k$; however, $l(u) = 1$. Now we consider two cases for $v$.

\begin{enumerate}
\item[Case 1.]If $v \in {\rm int C}$, then we can write $v$ as a conic combination of $a_1, a_2, \ldots, a_k$ with positive coefficients:

$$v = 2 \sum_{i=2}^{l} \alpha_i b_i = \sum_{i=1}^{k} c_i a_i, $$

\noindent where $c_i > 0$ for all $i$. 

\item[Case 2.] If $v \notin {\rm int} C$. let $L = \{\theta u + (1-\theta) v \; | \; \theta \in [0,1]\}$. Since $L$ intersects the interior of $C$ at $x$ and $\{u, v\} \notin {\rm int} C$, there exists $u',v'$ such that $L \cap {\rm bd} C = \{u',v'\}$. Suppose $v'$ is on the line segment between $v$ and $x$ (see Figure~\ref{fig:cone}). Let $L'=  \{\theta u + (1-\theta) v' \; | \; \theta \in [0,1]\}$ and note that $x \in L'$. Since $v' \in {\rm bd} C$, it can be written as conic combination of at most $k-1$ of $a_1, \ldots, a_k$. Without loss of generality assume that $v' = \sum_{i=2}^{k}{\beta_i a_i}$. For some $\theta \in (0,1)$:
$$x = \theta u + (1-\theta) v' =  \alpha_1' b_1 + \sum_{i=2}^{k}{\beta_i' a_i},$$

where $\alpha_1' = 2 \theta \alpha_1$ and $\beta'_i = (1 -\theta) \beta_i$. Using the representation in \eqref{k-rep}, we get:
\begin{align*}
\alpha_1' b_1 &= \gamma_1 a_1 +  \sum_{i=2}^{k} (\gamma_i - \beta_i') a_i. \\
\end{align*}

We have $l(\alpha_1' b_1) = 1$, and by Lemma \ref{arbit-sum-primal}, $K(\alpha_1' b_1) = 1 + |\{i | \gamma_i \neq \beta_i', i = 2, \ldots, k \}|$. Therefore, $\gamma_i = \beta_i'$ for all $i = 2,\ldots,k$ and $b_1 = a_1$. Combining the previous fact with \eqref{k-rep} and \eqref{l-rep}, we get:
  \begin{align}
x - \alpha_1 a_1 = (\gamma_1 - \alpha_1) a_1 + \sum_{i=2}^{k} \gamma_i a_i  = \sum_{i=2}^{l} \alpha_i b_i.
\end{align}

If $\gamma_{1} = \alpha_1$, by the induction hypothesis  $k = l$ , which is a contradiction. Now, suppose $\gamma_1 - \alpha_1 \neq 0$. In both cases we produced a point $y = v$ such that $K(y) = k$ and $l(y) \leq l-1$. We can continues this procedure until we get a $y$ such that $K(y) = k$ and $l(y) = 1$, which gives us the contradiction.
\qed

\end{enumerate}
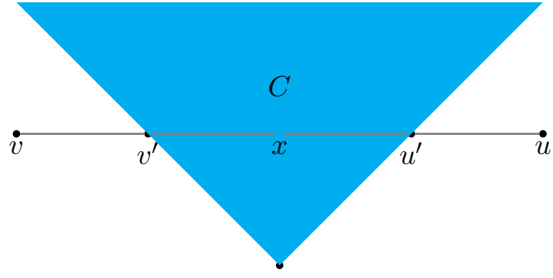
\begin{figure}
\begin{center}
	\begin{tikzpicture}
	[scale=3.5,font = \large, dot/.style={circle,draw=black!100,fill=black!100,thick,inner sep=0pt,minimum size=2pt}]
	    \node at (-1,1) (n1) {};
	    \node[dot] at (0,0)  (n2) {};
	    \node at (1,1)  (n3) {};
 	    \node[dot] at (1,0.5)  (e1){};
 	   \node[dot] at (0.5,0.5) (e2) {};
 	   \node[dot] at (-0.5,0.5) (e3) {};
 	   \node[dot] at (-1, 0.5) (e4) {};
 	  \node[dot] at (0, 0.5) (x){};
             \fill[color=cyan]
    (-1,1)  -- (0,0) -- (1,1) -- cycle;
	    \draw[gray,thick] (e1) -- (e2) -- (x) -- (e3) -- (e4);
	    \node[below] at (e1) {$u$};
	    \node[below] at (e2) {$u'$};
	    \node[below] at (x) {$x$};
	     \node[below] at (e3) {$v'$};
	    \node[below] at (e4) {$v$};
	     \node[below] at (0,0.75) {$C$};
	\end{tikzpicture}
	\end{center}
	\caption{Relative position of $u'$ and $v'$ on the line segment between $u$ and $v$. 
	}\label{fig:cone}
	\end{figure}

\end{proof}
\subsection{Proof of Proposition \ref{convresult}}

In iteration $t+1$ when the backtrack procedure stops, the following inequality holds true:
\begin{align}\notag
\phi_{\lambda}(x^{(t+1)})\leq m_{M_{t+1}}(x^{(t)},x^{(t+1)}) &= \min_{x}{f(x^{(t)}) + \langle \nabla f(x^{(t)}) , x-x^{(t)} \rangle + \frac{M_{t+1}}{2} \normu{x-x^{(t)}}^2 + \lambda \norm{x}}\\ \label{backtrack-ineq}
&\leq \min_{x} \phi_{\lambda}(x) + \frac{M_{t+1}}{2} \normu{x-x^{(t)}}^2.
\end{align}

On the other hand, by \eqref{lip1-itr}, we have

$$\phi_{\lambda}(x^{(t+1)})\leq m_{L_f}(x^{(t)},x^{(t+1)}), $$

\noindent which ensures $M_{t+1} \leq \gamma_{\rm inc} L_f$ since $m_{L}(x^{(t)},x^{(t+1)})$ is non-decreasing in $L$. By \eqref{str1-itr}, we have:
\begin{equation}\label{StrongOpt}
\phi_{\lambda}(x^{(t)}) \geq \phi_{\lambda}(x^{*}) + \frac{\mu_f}{2} \normu{x^{(t)}-x^{*}}^2.
\end{equation}
If we confine $x$ to $\{\alpha x^{*}+(1-\alpha)x^{(t)}\,|\, 0 \leq \alpha \leq 1\}$, inequality \eqref{backtrack-ineq} combined with \eqref{StrongOpt} results in
\begin{align*}
\phi_{\lambda}(x^{(t+1)}) &\leq \min_{\alpha \in [0 ,1]}\{{\phi_{\lambda}(\alpha x^{*}+(1-\alpha)x^{(t)})+\frac{\alpha^2 M_{t+1}}{2}{\normu{x^{(t)}-x^{*}}^2}}\}\\
&\leq\min_{\alpha \in [0 ,1]}\{{\alpha \phi_{\lambda}(x^{*})+(1-\alpha)\phi_{\lambda}(x^{(t)})+\frac{\alpha^2 M_{t+1}}{2}{\normu{x^{(t)}-x^{*}}^2}}\}\\
&\leq\min_{\alpha \in [0 ,1]}\{{\alpha \phi_{\lambda}(x^{*})+(1-\alpha)\phi_{\lambda}(x^{(t)})+\frac{\alpha^2 \gamma_{\rm inc} L_f}{ \mu_f}(\phi_{\lambda}(x^{(t)})-\phi_{\lambda}(x^{*}))}\}.
\end{align*}

      The RHS of the above inequality is minimized for $\alpha^* = \min\{1,\frac{\mu_f}{2 \gamma_{\rm inc} L_f}\}$. Therefore, we get
       $$ \phi_{\lambda}(x^{(t+1)}) - \phi_{\lambda}(x^{*}) \leq (1-\alpha^* + \frac{{\alpha^*}^2 \gamma_{\rm inc} L_f}{\mu_f}) (\phi_{\lambda}(x^{(t)})-\phi_{\lambda}(x^{*}))\leq (1-\frac{\mu_f}{4 \gamma_{\rm inc} L_f}) (\phi_{\lambda}(x^{(t)})-\phi_{\lambda}(x^{*})).$$
       
%
    To prove \eqref{second-co}, we note that the backtrack stopping criteria ensures
    \begin{align}\notag
\phi_{\lambda}(x^{(t+1)})&\leq {f(x^{(t)}) + \langle \nabla f(x^{(t)}) , x^{(t+1)}-x^{(t)} \rangle + \frac{M_{t+1}}{2} \normu{x^{(t+1)}-x^{(t)}}^2 + \lambda \norm{x^{(t+1)}}}\\  \notag
&\leq {f(x^{(t)}) - \langle M_{t+1} (x^{(t+1)}-x^{(t)}) +\xi , x^{(t+1)}-x^{(t)} \rangle + \frac{M_{t+1}}{2} \normu{x^{(t+1)}-x^{(t)}}^2 + \lambda \norm{x^{(t+1)}}}\\ \notag
&\leq {f(x^{(t)}) - \frac{M_{t+1}}{2} \normu{x^{(t+1)}-x^{(t)}}^2 +\langle \xi , x^{(t)} - x^{(t+1)} \rangle +\lambda \norm{x^{(t+1)}}}\\ 
&\leq \phi_{\lambda}(x^{(t)}) - \frac{M_{t+1}}{2} \normu{x^{(t+1)}-x^{(t)}}^2.\label{omega-2}
\end{align}

The hypothesis \eqref{str2-itr} ensures $M_{t+1} \geq \mu_f$. Combining  \eqref{omega-1} and \eqref{omega-2} and using the lower and the upper bounds on $M_{t+1}$, we get the desired result
\begin{align*}
    \omega_{\lambda}(x^{(t+1)}) 
    &\leq  \dnorm{M_{t+1} (x^{(t)} - x^{(t+1)}) +\nabla f(x^{(t+1)})- \nabla f(x^{(t)})}\\
    &\leq \theta ({M_{t+1}} + {L'_f}) \normu{x^{(t+1)} - x^{(t)}} \\
    &\leq \theta (1 + \frac{L'_f}{M_{t+1}}) \sqrt{2 M_{t+1}(\phi_{\lambda}(x^{(t)})-\phi_{\lambda}(x^{(t+1)}))} \\
    &\leq \theta (1+\frac{L'_f}{\mu_f})\sqrt{ 2 \gamma_{\rm inc}L_f (\phi_{\lambda}(x^{(t)})-\phi_{\lambda}(x^*)) }.
 \end{align*}


\subsection{Proof of Lemma \ref{conv-lemma-1}}

By the hypothesis  there exists $\xi \in \partial \norm{x}$ such that $\dnorm{A^*(A x-b)+\lambda \xi} \leq \delta\lambda$. Therefore, we can write
\begin{align} \notag
\delta \lambda \norm{x-x_{0}} &\geq \norm{x-x_{0}} \dnorm{A^*(A x-b)+\lambda \xi} \geq \langle (x-x_{0}) , A^*(A x-b)+\lambda \xi \rangle\\ \notag
  &= \langle (x-x_{0}) , A^*(A(x-x_{0}))-A^* z+\lambda \xi \rangle \\ \notag
  &= \normu{A(x-x_{0})}^2 - \langle x-x_{0} , A^* z \rangle +\lambda \langle x-x_{0} ,  \xi \rangle\\ \label{throughyoureyes}
   &\geq \normu{A(x-x_{0})}^2 - \norm{x-x_{0}} \dnorm{A^* z} +\lambda (\norm{x} - \norm{x_{0}}).
      \end{align}
      
      Now we lower-bound $\norm{x}$:
      $$\norm{x} = \norm{x-x_0+x_0} \geq \norm{\ptpx{x_{0}}{x-x_{0}}+x_{0}} -  \norm{\ptx{x_{0}}{x-x_{0}}}. $$

      By Lemma \ref{dual}, there exists $s \in T_{x_{0}}^{\bot}$ such that $\innerB{s}{\ptpx{x_{0}}{x-x_{0}}} = \norm{\ptpx{x_{0}}{x-x_{0}}}$ and $\dnorm{s} = 1$. 
Note that $e_{x_0}+s \in \partial{\norm{x_{0}}}$ hence $\dnorm{e_{x_0}+s} \leq 1$. Therefore, we get:
      $$\norm{\ptpx{x_{0}}{x-x_{0}}+x_{0}} \geq \innerB{e_{x_0}+s}{\ptpx{x_{0}}{x-x_{0}}+x_{0}} \geq \norm{\ptpx{x_{0}}{x-x_{0}}}+\norm{x_{0}},$$ 
      
      \begin{equation}\label{that'sthesprit}
      \norm{x} - \norm{x_{0}} \geq \norm{\ptpx{x_{0}}{x-x_{0}}} -  \norm{\ptx{x_{0}}{x-x_{0}}} .
      \end{equation}
      
Combining \eqref{that'sthesprit} and \eqref{throughyoureyes}, we get      
\begin{align*}
     \delta \lambda \norm{x-x_{0}} &\geq  \lambda (\norm{\ptpx{x_{0}}{x-x_{0}}} -  \norm{\ptx{x_{0}}{x-x_{0}}})- \norm{x-x_{0}} \dnorm{A^* z} +\normu{A(x-x_{0})}^2. 
      \end{align*}       
      
      By applying triangle inequality to $\norm{x-x_0}$, we obtain
      \begin{align}\label{whoknowswhat}
     (\lambda(1+ \delta) +\dnorm{A^* z})\norm{\ptx{x_{0}}{x-x_{0}}}&\geq (\lambda(1-\delta) -\dnorm{A^* z})\norm{\ptpx{x_{0}}{x-x_{0}}}+\normu{A(x-x_{0})}^2.
      \end{align}

That yields
      \begin{align}\notag
     \frac{\norm{x-x_0}}{\normu{x-x_0}} &\leq \frac{\norm{\ptx{x_{0}}{x-x_{0}}}+\norm{\ptpx{x_{0}}{x-x_{0}}}}{\normu{\ptx{x_{0}}{x-x_{0}}}}\\ \notag
&\leq (1+\gamma) \frac{\norm{\ptx{x_{0}}{x-x_{0}}}}{\normu{\ptx{x_{0}}{x-x_{0}}}}\leq (1+\gamma)\sqrt{ck_{0}}.
      \end{align}

Using the definition of the lower restricted isometry constant, we derive
  \begin{align*}
         \rho_{-}(A,{ c(1+\gamma)^2 k_{0}})\normu{x-x_{0}}^2 \leq \normu{A(x-x_{0})}^2 &\leq^{\eqref{whoknowswhat}} {((1+\delta)\lambda + \dnorm{A^* z})}\norm{\ptx{x_{0}}{x-x_{0}}} \\
         & \leq {\sqrt{ck_{0}}((1+\delta)\lambda + \dnorm{A^* z})}\normu{\ptx{x_{0}}{x-x_{0}}}\\
&\leq {\sqrt{ck_{0}}((1+\delta)\lambda + \dnorm{A^* z})}\normu{{x-x_{0}}},
  \end{align*}

\noindent  which yields the following bounds  
  \begin{align} \label{H-first}
& \normu{x-x_{0}} \leq \frac{{\sqrt{ck_{0}}((1+\delta)\lambda + \dnorm{A^* z})}}{\rho_{-}(A,{ c(1+\gamma)^2 k_{0}})},\\ \label{H-second}
& \norm{x-x_{0}} \leq \frac{{ck_{0}(1+\gamma)((1+\delta)\lambda + \dnorm{A^* z})}}{\rho_{-}(A,{ c(1+\gamma)^2 k_{0}})}.
\end{align}

By convexity of $\phi_{\lambda}$,
\begin{align*}
\phi_{\lambda}(x)-\phi_{\lambda}(x_0)  \leq  \innerB{\lambda\xi+A^*(A x-b)}{x-x_0} \leq \frac{{ck_{0}\delta \lambda(1+\gamma)((1+\delta)\lambda + \dnorm{A^* z})}}{\rho_{-}(A,{ c(1+\gamma)^2 k_{0}})}.
\end{align*}



\subsection{Proof of Lemma \ref{conv-lemma-2}}
Let $\Delta=\frac{3 ck_{0} \lambda(1+\gamma)}{2 \rho_{-}(A,{ c(1+\gamma)^2 k_{0}})}$. We can write
\begin{align}\notag
&\phi_{\lambda}(x) \leq \phi_{\lambda}(x_{0})+ \delta \lambda \Delta 
\end{align}
\begin{align}\notag
\Rightarrow\quad   \frac{1}{2} \normu{A x-b}^2  - \frac{1}{2} \normu{Ax_{0}-b}^2  &\leq  \lambda (\norm{x_{0}} - \norm{x}) + \delta \lambda \Delta\\ \label{int-ineq}
&\leq  \lambda \norm{x_{0}-x}+ \delta \lambda \Delta
\end{align}

 If $\norm{x-x_{0}} \leq \Delta$, half of the conclusion is immediate. To get the second half, we can expand the left hand side of \eqref{int-ineq} to get: 
\begin{align*}
 \frac{1}{2} \normu{A(x-x_{0})}^2 &\leq  \lambda \norm{x-x_0} +\langle x-x_{0} , A^* z\rangle +  \delta \lambda\Delta\\
& \leq  (\lambda+\dnorm{A^*z}) \norm{x-x_0}+  \delta \lambda \Delta\\
& \leq  (\frac{5}{4}+\delta)\lambda \Delta \leq \lambda \frac{3 \Delta}{2}.
\end{align*}

 Suppose $\norm{x-x_{0}} > \Delta$, then from \eqref{int-ineq} we get:

\begin{align*}f
 \lambda (\norm{x}-\norm{x_{0}}) &\leq  \frac{1}{2} \normu{A x_{0}-b}^2 - \frac{1}{2} \normu{A x-b}^2 + \delta \lambda \norm{x-x_{0}}\\
 &\leq  -\frac{1}{2} \normu{A(x-x_{0})}^2 + \langle x-x_{0} , A^* z\rangle +  \delta \lambda \norm{x-x_{0}}\\
&\leq  -\frac{1}{2} \normu{A(x-x_{0})}^2+\dnorm{A^* z}\norm{ x-x_{0}}   +  \delta \lambda \norm{x-x_{0}}.
\end{align*}

By using \eqref{that'sthesprit} and triangle inequality we get:
 \begin{align*}
     (\lambda(1+ \delta) +\dnorm{A^* z})\norm{\ptx{x_{0}}{x-x_{0}}}&\geq (\lambda(1-\delta') -\dnorm{A^* z} )\norm{\ptpx{x_{0}}{x-x_{0}}}+\frac{1}{2}\normu{A(x-x_{0})}^2 .
    \end{align*}     
     Using the same reasoning as in the proof of Lemma \ref{conv-lemma-1}, we get the desired results.


\subsection{Proof of Lemma \ref{conv-lemma-3}}

By first order optimality condition there exists $\xi \in \partial \norm{x^{+}}$ such that:
\begin{align*}
\lambda \xi &= L(x-x^{+}) - \nabla f(x) \\
& =  L(x-x^{+}) - A^*(A x-b) \\
& =  L(x-x^{+}) - A^*(A (x-x_0)) + A^* z
\end{align*}

Note that $\xi = e_{x^{+}} + v$ for some $v \in T_{x^{+}}^{\bot}$. By Lemma~\ref{dual}, there exists $v' \in T_{x^{+}}^{\bot} \cap \mathcal{B}_{\dnorm{\cdot}}$ such that $\inner{v'}{v} = \norm{v}$. Since $e_{x^+} + v' \in \partial \norm{x^{+}}$, $\dnorm{e_{x^+} + v'} \leq 1$. Therefore, we can write:

\begin{align*}
\norm{\xi} = \norm{e_{x^+} + v} \geq \inner{e_{x^+} + v'}{e_{x^+} + v} = \norm{e_{x^+}} + \norm{v}
\Rightarrow K(x^{+}) = \norm{e_{x^+}} \leq \norm{\xi}.
\end{align*}


Let $\xi= \sum_{i=1}^{l}\gamma_i a_i$, where $a_1,\ldots,a_{l}$ and $\gamma_1, \ldots, \gamma_{l}$ are given by the orthogonal representation theorem. Since $\gamma_i \leq 1$ for all $i$, $l \geq \norm{\xi}$. If $\norm{\xi} > \tilde{k}$, we can define $u = \sum_{i=1}^{\tilde{k}}{a_i}$, then
\begin{align}\notag
\tilde{k} \lambda \leq \innerB{u}{\lambda \xi} &=\innerB{u}{L(x^{+}-x)} - \innerB{A u}{A(x-x_0)} + \innerB{u}{A^* z}\\ \notag 
&\leq  L \norm{x^{+}-x} + \sqrt{\rho_{+}(A,\tilde{k}) \tilde{k}} \normu{A(x-x_0)}+ \tilde{k} \dnorm{A^* z}\\\label{H-first-a}
 \Rightarrow  \frac{3\tilde{k}\lambda}{4} &\leq L \norm{x^{+}-x} + \sqrt{\rho_{+}(A,\tilde{k}) \tilde{k}} \normu{A(x-x_0)}.
\end{align}

Since $\phi_{\lambda}(x^{+}) \leq \phi_{\lambda}{(x)}$, by Lemma \ref{conv-lemma-2}, we have:
%
\begin{align*}
\norm{x^{+}-x} \leq \norm{x^{+}-x_0}+\norm{x-x_0} \leq \frac{9 ck_{0} \lambda(1+\gamma)}{ \rho_{-}(A,{ c(1+\gamma)^2 k_{0}})},\\
\normu{A(x-x_0)}^2 \leq  \frac{9 ck_{0} \lambda^2 (1+\gamma)}{ \rho_{-}(A,{ c(1+\gamma)^2 k_{0}})}.
\end{align*}

Define 
\begin{align*}
\alpha = \gamma_{\rm inc} \rho_{+}(A,2\tilde{k})\frac{9 ck_{0} (1+\gamma)}{ \rho_{-}(A,{ c(1+\gamma)^2 k_{0}})},\\
\beta^2 =   \rho_{+}(A,\tilde{k})\frac{9 ck_{0} (1+\gamma)}{ \rho_{-}(A,{ c(1+\gamma)^2 k_{0}})}.
\end{align*}

We can rewrite \eqref{H-first-a} as:
\begin{align*}
\frac{3\tilde{k}}{4} -\alpha -\beta \sqrt{\tilde{k}} < 0 \Rightarrow \sqrt{\tilde{k}} < \frac{2}{3} (\beta+ \sqrt{\beta^2 + 3 \alpha}) \leq 2 \sqrt{\alpha}. 
\end{align*}

But this contradicts Assumption~\ref{assumption-2}, so $\norm{\xi} \leq  \tilde{k}$ hence $K(x^{+}) \leq \tilde{k}$.